\documentclass[reqno, 10pt]{amsart}
\usepackage{amssymb,mathrsfs}
\usepackage[usenames, dvipsnames]{color}
\usepackage{esint}
\usepackage{hyperref}
\usepackage{todonotes}
\usepackage{verbatim}
\usepackage{enumerate}

\usepackage{cite}
\usepackage[nobysame,non-sorted-cites]{amsrefs}
\def\MR#1{}

\theoremstyle{plain}
\newtheorem{theorem}{Theorem}[section]
\newtheorem{lemma}[theorem]{Lemma}
\newtheorem{corollary}[theorem]{Corollary}
\newtheorem{proposition}[theorem]{Proposition}

\theoremstyle{definition}

\theoremstyle{remark}
\newtheorem{remark}[theorem]{Remark}

\newcommand{\dv}{\operatorname{div}}

\numberwithin{equation}{section}

\newcommand{\bN}{\mathbb{N}}

\newcommand{\bR}{\mathbb{R}}

\newcommand{\bS}{\mathbb{S}}

\newcommand\cD{\mathcal{D}}

\newcommand\sA{\mathscr{A}}

\makeatletter
\def\dashint{\operatorname%
{\,\,\text{\bf--}\kern-.98em\DOTSI\intop\ilimits@\!\!}}
\makeatother

\newcommand\dashnorm[2]{\nparallel\kern-.2em #1 \Vert_{#2}}

\begin{document}
\title[Conductivity problem with imperfect bonding interfaces]{Gradient estimates for the conductivity problem with imperfect bonding interfaces}

\author[H. Dong]{Hongjie Dong}

\author[Z. Yang]{Zhuolun Yang}

\author[H. Zhu]{Hanye Zhu}

\address[H. Dong]{Division of Applied Mathematics, Brown University, 182 George Street, Providence, RI 02912, USA}
\email{Hongjie\_Dong@brown.edu}

\address[Z. Yang]{Division of Applied Mathematics, Brown University, 182 George Street, Providence, RI 02912, USA}
\email{zhuolun\_yang@brown.edu}

\address[H. Zhu]{Department of Mathematics, Duke University, 120 Science Drive, Durham, NC 27708, USA}
\email{hanye.zhu@duke.edu}

\thanks{H. Dong was partially supported by the NSF under agreements DMS-2055244 and DMS-2350129.}
\thanks{Z. Yang was partially supported by the AMS-Simons Travel Grant.}
\thanks{H. Zhu was partially supported by the NSF under agreement DMS-2055244.}



\begin{abstract}
We study the field concentration phenomenon between two closely spaced perfect conductors with imperfect bonding interfaces of low conductivity type. The boundary condition on these interfaces is given by a Robin-type boundary condition.  We discover a \textit{new} dichotomy for the field concentration depending on the bonding parameter $\gamma$. Specifically, we show that the gradient of solution is uniformly bounded independent of $\varepsilon$ (the distance between two inclusions) when $\gamma$ is sufficiently small. However, the gradient may blow up when $\gamma$ is large. Moreover, we  identify the threshold of $\gamma$ and the optimal blow-up rates under certain symmetry assumptions.  The proof relies on a crucial anisotropic gradient estimate in the thin neck between two inclusions. We develop a general framework for establishing such estimate, which is applicable to a wide range of elliptic equations and boundary conditions.
\end{abstract}

\maketitle

\section{Introduction and Main results}

Our study is motivated by the phenomenon of field concentration in high-contrast composites, a central topic in composite material research.  When two inclusions within a composite material are positioned very close to each other and their material properties significantly differ from the surrounding matrix, it may lead to a significant build-up of stress or field intensity in the narrow gap between the inclusions. Understanding the field concentration phenomenon quantitatively is important since it may cause material failure (see, for example, \cites{BASL,Kel}). In some other cases, the inclusions are designed to create the field concentration to achieve desired enhancement of the field (see, for example, \cites{acsphotonics,KANG20191670}). Theoretical analysis related to the field concentration phenomenon originated from the effective medium theory. In the pioneering work \cite{Keller}, the electrical or thermal resistance of the narrow gaps between the inclusions was analyzed in order to estimate the effective conductivity of a medium containing a dense (nearly touching) array of perfectly conducting or insulating cylinders (2D) or perfectly conducting spheres (3D), all with perfect interfaces. See also \cite{Keller2} in the setting of elasticity. There have been significant progress over the past three decades in quantitatively analyzing the field concentration phenomenon in scenarios where the inclusions and the background matrix are perfectly bonded: optimal estimates for the local fields in the narrow gaps between the inclusions and asymptotic characterizations capturing the field concentration phenomenon have been derived.

Motivated by the significant physics work \cite{prl} and recent pioneering paper \cite{fukushima2024finiteness}, we investigate the case of non-ideal bonding. This model holds significant practical relevance,  as it takes into account the contact resistance due to roughness and possible debonding at the interface, and it also approximates the membrane structure in biological systems. It is known in physics literature that interfacial effects are very important in a variety of systems. For example, they may dramatically alter the effective properties (\cites{prl,LipVer,CHIEW198390,BenY,EVERY1992123,HASHIN1992767}). In \cites{prl, LipVer}, lower and upper bounds for the effective conductivity of a medium containing randomly packed equisized spheres with imperfect bonding interfaces were deduced by constructing trial fields that account for the complex interactions between the spheres and using minimum energy principles. It is noteworthy that the work \cites{prl,LipVer} could not provide estimates for the local fields or capture the field concentration phenomenon when the spheres are closely located, because, as stated in \cite{prl}, ``the complexity of the microstructure prohibits one from obtaining the local fields exactly". However, in this paper, we will establish optimal estimates for the local fields under certain symmetry assumptions.

Let us describe the mathematical setup.  Since the region of interest is the local narrow area in between two inclusions which are closely located,  the mathematical problem is formulated with two inclusions. Let $n\ge 2$, $\Omega \subset \bR^n$ contain two relatively strictly convex open sets $D_{1}$ and $D_{2}$ with dist$(D_1 \cup D_2, \partial \Omega) > 1$. Let
$$\varepsilon: = \mbox{dist}(D_1, D_2)$$
and $\widetilde{\Omega} := \Omega \setminus \overline{(D_1 \cup D_2)}$. When the inclusions $D_1, D_2$ and the background matrix are perfectly bonded, the conductivity problem can be modeled by the following elliptic equation:
\begin{equation}\label{general_problem}
\begin{cases}
\mathrm{div}\Big(a(x)\nabla{u}\Big)=0&\mbox{in}~\Omega,\\
u=\varphi(x)&\mbox{on}~\partial\Omega,
\end{cases}
\end{equation}
where $a$ represents the conductivity distribution, given by
$$
a = k_1 \,\chi_{D_1} + k_2 \,\chi_{D_2} + \chi_{\widetilde{\Omega}},
$$
$u$ represents the voltage potential, and $-\nabla u$ stands for the electric field. A key feature of the perfectly bonded problem is the continuity of potential and the continuity of flux across the interfaces:
\begin{equation}\label{transmission}
u |_+ = u|_- \quad \mbox{and}\quad \left.\frac{\partial{u}}{\partial\nu} \right|_+ = k_i \left.\frac{\partial{u}}{\partial\nu} \right|_- \quad \mbox{on}~\partial D_i.
\end{equation}
Here and throughout the paper, the subscripts
$\pm$ indicate the limit from outside and inside the inclusion, respectively, and $\nu$ denotes the inner normal vector on $\partial {D}_{1} \cup \partial {D}_{2}$. 
When $k_i$'s are bounded away from $0$ and infinity, it was proved in \cites{BV,LV,LN} that the gradient of solutions remains bounded independent of $\varepsilon$. However, if $k_i$'s degenerate to either $\infty$ or $0$, $\nabla u$ may blow up as $\varepsilon \to 0$. Specifically, in the perfect conducting case when $k_1 = k_2 = \infty$,  equation \eqref{general_problem} becomes
\begin{equation}\label{perfect}
\begin{cases}
\Delta u =0&\mbox{in}~\widetilde\Omega,\\
u  = U_j &\mbox{on}~\partial D_j, j = 1,2,\\
\int_{\partial D_j} \partial_\nu u \, d\sigma = 0& j=1,2,\\
u=\varphi&\mbox{on}~\partial\Omega,
\end{cases}
\end{equation}
where $U_j$ is some constant determined by the third line.
 It was known that
\begin{equation*}
\begin{cases}
\| \nabla u \|_{L^\infty(\widetilde{\Omega})} \le C\varepsilon^{-1/2} \|\varphi\|_{C^2(\partial \Omega)} &\mbox{when}~n=2,\\
\| \nabla u \|_{L^\infty(\widetilde{\Omega})} \le C|\varepsilon \ln \varepsilon|^{-1} \|\varphi\|_{C^2(\partial \Omega)} &\mbox{when}~n=3,\\
\| \nabla u \|_{L^\infty(\widetilde{\Omega})} \le C\varepsilon^{-1} \|\varphi\|_{C^2(\partial \Omega)} &\mbox{when}~n\ge 4;
\end{cases}
\end{equation*}
see \cites{AKLLL,AKL,Y1,Y2,BLY1,BLY2}. These bounds were shown to be optimal. In the insulated case where $k_1 = k_2 = 0$, equation \eqref{general_problem} becomes
\begin{equation}\label{insulated}
\begin{cases}
\Delta u =0&\mbox{in}~\widetilde\Omega,\\
\partial_\nu u  = 0 &\mbox{on}~\partial D_j, j = 1,2,\\
u=\varphi&\mbox{on}~\partial\Omega.
\end{cases}
\end{equation}
 While the optimal blow-up rate in two dimensions was captured about two decades ago in \cites{AKL,AKLLL}, the optimal rate in dimensions $n \ge 3$ was only recently identified in \cites{DLY,DLY2,li2024optimalgradientestimatesinsulated}. This optimal rate is linked to the first non-zero eigenvalue of an elliptic operator on $\bS^{n-2}$, which is determined by the principal curvatures of the inclusions. When $k_1 = 0$ and $k_2 = \infty$, it was recently shown in \cites{JiKang,DonYan23} that $\nabla u$ is bounded independent of $\varepsilon$. Thus, the study of field concentration phenomena, particularly regarding the optimal blow-up rate, when the bonding is perfect, is relatively comprehensive. For a detailed review of related work in this area, we refer the reader to the survey article \cite{Kang}.

 When the bonding between the inclusions and the background matrix is non-ideal, at least one of the transmission conditions \eqref{transmission} fails. We consider the inclusion with imperfect bonding interfaces of low conductivity type (LC-type), which can be understood as an approximation of an inclusion with a low conductivity shell as the shell thickness approaches zero. Specifically, consider an inclusion with a core-shell structure, where the shell has a thickness denoted by $t$. Let $k$ and $k_s$ be the conductivities of the core and shell, respectively. As $t \to 0$, if
$$
\gamma^{-1}: = \lim_{t \to 0} \frac{k_s}{t}
$$
exists and is positive, then the transmission conditions become
$$
\left.\frac{\partial{u}}{\partial\nu} \right|_+ = k \left.\frac{\partial{u}}{\partial\nu} \right|_- =- \gamma^{-1} (u |_+ - u|_- ),
$$
and we say that the inclusion has an imperfect bonding interface of LC-type, where $\gamma$ is called the bonding parameter. In this scenario, the continuity of potential no longer holds. It is important to note that there is another type of imperfect bonding interface known as the high conductivity type (HC-type), where the continuity of flux fails. In two dimensions, The HC-type problem is indeed the dual problem of the LC-type problem. For a detailed discussion and derivation of the transmission conditions, we refer the reader to \cites{fukushima2024finiteness,BenMil}.

Throughout this article, we consider $D_1$, $D_2$ to have imperfect bonding interface of LC-type, and let $k_1 = k_2 = \infty$. Then the conductivity problem becomes
\begin{equation}\label{eq1}
\begin{cases}
\Delta u =0&\mbox{in}~\widetilde\Omega,\\
u + \gamma \partial_\nu u = U_j &\mbox{on}~\partial D_j, j = 1,2,\\
\int_{\partial D_j} \partial_\nu u \, d\sigma = 0& j=1,2,\\
u=\varphi&\mbox{on}~\partial\Omega,
\end{cases}
\end{equation}
where $\nu$ is the inward normal vector on $\partial D_j$,  $U_j$ is some constant determined by the third line of \eqref{eq1}.
 This setting corresponds to the physical situation investigated in \cites{prl, ar} --  suspensions of metallic particles in epoxy matrices with interfacial resistance of Kapitza type. The solution $u\in H^{1}
(\widetilde\Omega)$ to equation \eqref{eq1} is indeed the minimizer of a functional in an appropriate function space (see Lemma \ref{lemma_EU}): $I[u] = \min_{v \in \sA}I[v]$, where
\begin{equation}\label{minimizer}
\begin{aligned}
\sA &:= \{ v \in H^{1}(\widetilde\Omega):\, v = \varphi~~\mbox{on}~~\partial\Omega \},\quad (v)_{\partial D_j}:= \fint_{\partial D_j} v\, d\sigma,\ j=1,2,\\
I[v] &:= \int_{\widetilde\Omega} |\nabla v|^2+\gamma^{-1}\int_{\partial D_1} |v-(v)_{\partial D_1}|^2+\gamma^{-1}\int_{\partial D_2} |v-(v)_{\partial D_2}|^2.
\end{aligned}
\end{equation} In \cite{fukushima2024finiteness}, it was proved that in two dimensions, the gradient of the solution to
\begin{equation}\label{eq2}
\begin{cases}
\Delta u =0&\mbox{in}~\bR^2\setminus \overline{(D_1 \cup D_2)},\\
u + \gamma \partial_\nu u = U_j &\mbox{on}~\partial D_j, j = 1,2,\\
\int_{\partial D_j} \partial_\nu u \, d\sigma = 0& j=1,2,\\
u - x_2 = O(|x|^{-1})&\mbox{as}~|x| \to \infty,
\end{cases}
\end{equation}
is bounded independent of $\varepsilon$ when $\gamma>0$ and $D_1, D_2$ are disks in $\bR^2$ with radius $R$, centered at $(0,-R-\varepsilon/2)$ and $(0,R+\varepsilon/2)$ respectively. It is well known in the literature \cite{AKL} that the gradient of solution to \eqref{eq2} with $\gamma = 0$ (the perfect conductivity problem) blows up as $\varepsilon \to 0$. This surprising result shows that a thin coating of low conductivity can prevent such blow-up. The authors of \cite{fukushima2024finiteness} conjectured that 
the gradient of the solution to \eqref{eq1} is always bounded independent of $\varepsilon$ for arbitrary $\gamma>0$ and boundary data $\varphi$, as a biological system cannot endure large stress.

In this article, we prove three main results:
\begin{enumerate}[1.]
\item \underline{Upper bound for gradient}: We establish in Theorem \ref{thm-1/2} an upper bound of order $\varepsilon^{-1/2}$ for the gradient of the solution to \eqref{eq1}. A direct consequence of this result, presented in Corollary \ref{cor2}, is that if $\widetilde\Omega$ is symmetric with respect to $x_n$, and the boundary value $\varphi$ is odd in $x_n$, then the gradient is bounded independent of $\varepsilon$. This generalizes the boundedness result in \cite{fukushima2024finiteness} to any dimension $n \ge 2$, more general inclusions $D_1, D_2$, and more general boundary data $\varphi$.
\item  \underline{Absence of field concentration for small $\gamma$}: We show that when the bonding parameter $\gamma$ is sufficiently small, the gradient of solution is bounded independent of $\varepsilon$. This result confirms the conjecture raised in \cite{fukushima2024finiteness} under the assumption that $\gamma$ is sufficiently small.
\item \underline{Dichotomy for field concentration}: We show that when $\gamma$ is large, the conjecture does not hold and the boundedness result in Corollary \ref{cor2} is highly unstable, in the sense that if the boundary data $\varphi$ is slightly tilted away from being odd in $x_n$ ($\varphi = \delta x_1 + x_n$ for some small $\delta$ for example), the gradient of the solution may blow up. More precisely, we show that when  $D_1$ and $D_2$ are balls of radius $R$, centered at $(0,-R-\varepsilon/2)$ and $(0,R+\varepsilon/2)$ respectively, $\Omega=B_{5R}$, and $\varphi = x_1$, there is a dichotomy for the field concentration phenomenon: the gradient of the solution to \eqref{eq1} is bounded when $0 < \gamma \le R$, but blows up when $\gamma > R$. Moreover, we establish an optimal gradient estimate when $\gamma > R$. Our analysis shows that the solution behaves similarly to that of the insulated conductivity problem in this case.
\end{enumerate}

We use the notation $x = (x', x_n)$, where $x' \in \bR^{n-1}$. We choose the coordinate so that near the origin, the part of $\partial D_1$ and $\partial D_2$, denoted by $\Gamma_+$ and $\Gamma_-$, are respectively the graphs of two $C^{1,1}$ functions in terms of $x'$. That is, for some $R_0 > 0$,
\begin{align*}
\Gamma_+ =& \left\{ x_n = \frac{\varepsilon}{2}+f(x'),~|x'|<R_0\right\},\\
\Gamma_- =& \left\{ x_n = -\frac{\varepsilon}{2}+g(x'),~|x'|<R_0\right\},
\end{align*}
where $f(x')$ and $g(x')$ are $C^{1,1}$ functions satisfying
\begin{equation}\label{fg_0}
f(0')=g(0')=0,\quad\nabla_{x'}f(0')=\nabla_{x'}g(0')=0,
\end{equation}
\begin{equation}\label{fg_1}
c_1 |x'|^2 \le f(x')-g(x')\quad\mbox{for}~~0\le |x'|<R_0,
\end{equation}
\begin{equation}\label{def:c_2}
    \|f\|_{C^{1,1}}\leq c_2, \quad\|g\|_{C^{1,1}}\leq c_2,
\end{equation}
with some positive constants $c_1$, $c_2$. For $x_0 \in \widetilde\Omega$ with $|x_0'|< R_0$, $0 < r\leq R_0 - |x_0'|$, we denote
\begin{align*}
\Omega_{x_0,r}:=\left\{(x',x_{n})\in \widetilde{\Omega}:~-\frac{\varepsilon}{2}+g(x')<x_{n}<\frac{\varepsilon}{2}+f(x'),~|x'-x_0' |<r\right\}
\end{align*}
and $\Omega_{r}:=\Omega_{0,r}$. For any domain $\mathcal{D}$, we denote the averaged $L^p$ norm of $u$ over $\mathcal{D}$ by
$$
\dashnorm{u}{L^p(\cD)}:= \left( \fint_{\cD} |u|^p \right)^{1/p}.
$$
Even though the constants $U_1$ and $U_2$ in \eqref{eq1} depend on $\varepsilon$, it can be shown that (see Lemma \ref{lem:bdd:U})
$$
\inf_{\partial \Omega} \varphi \le U_1, U_2 \le \sup_{\partial \Omega} \varphi.
$$
Therefore, by the maximum principle and classical gradient estimates (c.f. \cite{MR3059278}*{Chapters 1 and 4}, the solution $u \in H^1(\widetilde\Omega)$ of \eqref{eq1} satisfies
\begin{equation}
\label{u_C1_outside}
\|u\|_{L^\infty(\widetilde\Omega)} + \| \nabla u\|_{{L^\infty}(\widetilde\Omega \setminus \Omega_{R_0/2})} \le C\|\varphi\|_{C^{1,1}(\partial \Omega)}.
\end{equation}
As such, we focus on the following problem near the origin:
\begin{equation}\label{main_problem_narrow}
\begin{cases}
\Delta u =0&\mbox{in}~\Omega_{R_0},\\
u + \gamma \partial_\nu u = U_1&\mbox{on}~\Gamma_+,\\
u + \gamma \partial_\nu u = U_2&\mbox{on}~\Gamma_-,
\end{cases}
\end{equation}
where $\nu$ is the unit normal vector pointing upward on $\Gamma_+$ and pointing downward on $\Gamma_-$.

Our first main result is the following pointwise gradient estimate of order $\varepsilon^{-1/2}$.

\begin{theorem}
\label{thm-1/2}
Let $f$, $g$ be $C^{1,1}$ functions satisfying \eqref{fg_0}--\eqref{def:c_2}, $n \ge 2$, $\gamma > 0$, $\varepsilon \in(0,R_0/4)$, and $u \in H^1({\Omega}_{R_0})$ be a solution of \eqref{main_problem_narrow} for some constants $U_1$ and $U_2$. Then there exists a constant $C>0$ depending only on $n$, $\gamma$, $R_0$, $c_1$, and $c_2$, such that for any $x\in \Omega_{R_0/2}$ and $r=\frac{1}{4}(\varepsilon+|x'|^2)^{\frac{1}{2}}$,
\begin{equation}\label{gradient-1/2}
|\nabla u(x)| \le  C \left( r^{-1} \dashnorm{u-(U_1 + U_2)/2}{L^2(\Omega_{x,r})}  + |U_1 - U_2| \right).
\end{equation}
\end{theorem}

A quick corollary of the theorem is the boundedness of the gradient under some additional symmetry assumptions. This generalizes one of the results in \cite{fukushima2024finiteness}.

\begin{corollary}\label{cor2}
For $n \ge 2$ and $\varepsilon \in(0,R_0/4)$, let $\Omega$, $D_1$, and $D_2$ be $C^{1,1}$ domains so that $\widetilde{\Omega}$ is symmetric with respect to $x_n$, and \eqref{fg_0}--\eqref{def:c_2} hold. Let $\varphi \in C^{1,1}(\partial \Omega)$ be an odd function in $x_n$,  $\gamma > 0$, and $u \in H^1(\widetilde\Omega)$ be the solution to \eqref{eq1}. Then there exists a constant $C>0$ depending only on $n$, $\gamma$, $R_0$, $\Omega$, $c_1$, and $c_2$ such that for any $x\in \widetilde\Omega$,
\begin{equation}\label{gradient-bdd}
|\nabla u(x)| \le  C \, \|\varphi\|_{C^{1,1}(\partial \Omega)}.
\end{equation}
\end{corollary}

It is important to note that the factor on the right-hand side of estimate \eqref{gradient-1/2} is $r$ instead of $r^2$, the latter is comparable to the height of the narrow region. A similar anisotropic estimate has served as a fundamental tool in studying both linear and nonlinear insulated conductivity problems \cites{LY2,DLY,DLY2,DYZ23,li2024optimalgradientestimatesinsulated}. A key step in establishing such estimate involves extending the solution beyond the upper and lower boundaries of the narrow region to reach a larger scale of order $r$. For the insulated conductivity problem \eqref{insulated}, where the boundary condition is the homogeneous Neumann boundary condition, this extension can be naturally achieved using even and periodic extensions after flattening the boundaries. To prove Theorem \ref{thm-1/2}, we use a special flattening map introduced by the same authors in \cite{DYZ23} that preserves normal vectors on $\Gamma_\pm$, and construct a delicate auxiliary function to manage the inhomogeneity of the boundary condition. The framework presented in this proof is very robust and can be applied to inhomogeneous Dirichlet or Neumann boundary conditions, as well as to a wide range of elliptic equations and systems.

Our second result concerns the local problem \eqref{main_problem_narrow} with sufficiently small $\gamma$. We show that $\nabla u$ is uniformly bounded independent of $\varepsilon$. Moreover, $\nabla u$ exhibits a polynomial decay when $U_1 = U_2$. Along with \eqref{u_C1_outside} and \eqref{bdd:U} below, this result confirms the conjecture raised in \cite{fukushima2024finiteness} for small $\gamma$, without any assumptions on the symmetry of the domain $\tilde{\Omega}$ or the boundary data $\varphi$.

\begin{theorem}\label{thm_polynomial_upperbound}
Let $f$, $g$ be $C^{1,1}$ functions satisfying \eqref{fg_0}--\eqref{def:c_2}, $n \ge 2$, $\varepsilon \in(0,R_0/4)$, and $u \in H^1({\Omega}_{R_0})$ be a solution of \eqref{main_problem_narrow} for some constants $U_1$ and $U_2$. Then for any $\beta  \ge 0$, there exists a positive constant $\gamma_0$, depending only on $n$, $R_0$, $c_1$, $c_2$, and $\beta$, such that if $\gamma \in (0, \gamma_0)$, then
\begin{equation}\label{grad_u_polyupperbound}
|\nabla u(x)| \le C \Big( (\sqrt\varepsilon + |x'|)^\beta \|u - (U_1 + U_2)/2\|_{L^\infty(\Omega_{R_0})} + |U_1 - U_2| \Big)  \quad \mbox{in}~ \Omega_{R_0/2},
\end{equation}
where $C > 0$ depends only on  $n$, $R_0$, $c_1$, $c_2$, $\beta$, and $\gamma$.
\end{theorem}

 As $\gamma \to 0$, the problem \eqref{eq1} formally converges to the perfect conductivity problem \eqref{perfect}, where $\nabla u$ may blow up as $\varepsilon \to 0$ (see \cites{AKL,BLY1}). This suggests that the constants $C$ above might go to $\infty$ as $\gamma$ approaches $0$.

For the next result, we further assume that $D_1, D_2$ to be $C^{2,\sigma}$ for some $\sigma \in (0,1)$. In addition to \eqref{fg_0}, the corresponding graph functions $f$ and $g$ will satisfy
\begin{equation}\label{fg_2}
f(x') - g(x') = \mu \, |x'|^2 + O(|x'|^{2+\sigma})\quad\mbox{for}~~0<|x'|<R_0,~ \mu > 0.
\end{equation}
We also assume that for some $1 \le j \le n-1$, the domain $\widetilde\Omega$ is symmetric with respect to $x_j$, and boundary data $\varphi \in C^{1,1}$ is odd in $x_j$. In this case, by the symmetry and the uniqueness of solutions (Lemma \ref{lemma_EU}), we know that the solution $u$ to \eqref{eq1} is odd in $x_j$, and satisfies
\begin{equation}\label{robin}
\begin{cases}
\Delta u =0&\mbox{in}~\widetilde\Omega,\\
u + \gamma \partial_\nu u = 0&\mbox{on}~\partial D_1 \cup \partial D_2,\\
\int_{\partial D_i} \partial_\nu u \, d\sigma = 0& i=1,2,\\
u= \varphi&\mbox{on}~\partial\Omega.
\end{cases}
\end{equation}
Again, we focus on the local problem of \eqref{robin}, that is, 
\begin{equation}\label{narrow_region}
\begin{cases}
\Delta u =0&\mbox{in}~\Omega_{R_0},\\
u + \gamma \partial_\nu u = 0&\mbox{on}~\Gamma_\pm,
\end{cases}
\end{equation}
where $\nu$ is the unit normal vector pointing upward on $\Gamma_+$ and pointing downward on $\Gamma_-$. We define
\begin{equation}
\label{alpha_n}
\alpha=\alpha(n,\gamma, \mu):= \frac {-(n-1)+\sqrt{(n-1)^2+4(n-2 +  2/(\mu\gamma))}}{2}.
\end{equation}
Note that $\alpha$ is always positive, and $\alpha <1$ if and only if $\gamma >  1/\mu$.

Our third result is as follows.

\begin{theorem}\label{thm_upperbound}
For $n \ge 2$, $\varepsilon \in (0,R_0/4)$, $\gamma > 0$, $\mu > 0$, let $f,g$ satisfy \eqref{fg_0} and \eqref{fg_2}, and let $u \in H^1(\Omega_{R_0})$ be a solution of \eqref{narrow_region} that is odd in $x_j$ and $\Omega_{R_0}$ is symmetric in $x_j$ for some $1 \le j \le n-1$. Then there exists a positive constant $C$, depending only on $n$, $\gamma$, $\mu$, $R_0$, $\|f\|_{C^{2,\sigma}}$, and $\|g\|_{C^{2,\sigma}}$, such that 
\begin{itemize}
\item when $0 < \gamma \le 1/\mu$,
\begin{equation}\label{grad_u_upperbound1}
|\nabla u(x)| \le C \|u\|_{L^\infty(\Omega_{R_0})} \quad \mbox{in}~ \Omega_{R_0/2},
\end{equation}
\item when $\gamma >1/\mu$,
\begin{equation}\label{grad_u_upperbound2}
|\nabla u(x)| \le C \|u\|_{L^\infty(\Omega_{R_0})} (\varepsilon + |x'|^2)^{\frac{\alpha-1}{2}} \quad \mbox{in}~  \Omega_{R_0/2},
\end{equation}
\end{itemize}
where $\alpha$ is given in \eqref{alpha_n}.
\end{theorem}

Estimate \eqref{grad_u_upperbound2} indicates that when $\gamma > 1/\mu$, the gradient of solution may blow up as $\varepsilon \to 0$. We provide an example to show that the blow up can actually happen and the blow-up rate is optimal as in the following theorem.

\begin{theorem}
\label{main_thm}
For $n \ge 2$, $\mu > 0$, $\gamma > 1/\mu$, $\varepsilon \in (0,1/(4\mu))$, let $\Omega = B_{5/\mu}$, $D_1 , D_2$ be balls of radius $1/\mu$ centered at $(0', 1/\mu + \varepsilon/2)$ and $(0', -1/\mu - \varepsilon/2)$, respectively, $\varphi(x) = x_1$, and let $u \in H^1(\widetilde\Omega)$ be the solution of \eqref{robin}. Then there exist positive constants $c$ and $C$, depending only on $n$, $\gamma$, and $\mu$, such that
\begin{equation}
\label{grad_u_lower_bound}
\|\nabla u\|_{L^\infty(\Omega_{c \sqrt\varepsilon})} \ge \frac{1}{C}\varepsilon^{\frac{\alpha-1}{2}}, 
\end{equation}
where $\alpha$ is given in \eqref{alpha_n}.
\end{theorem}

\begin{remark}
The fact that $D_1$ and $D_2$ are balls of same radius is used in Step 3 of the proof of Theorem \ref{main_thm} in Section \ref{sec6}, so that $x_1$ is a subsolution of \eqref{robin}. This conclusion remains valid even if $D_1$ and $D_2$ are balls of radii $r_1$ and $r_2$ centered at $(0',r_1+\varepsilon/2)$ and $(0',-r_2-\varepsilon/2)$, respectively, under a stronger assumption that $\gamma > \max\{r_1,r_2\}$. More generally, if $D_1$ and $D_2$ are only assumed to be strictly convex smooth sets that are axially symmetric with respect to $x_n$, and $\widetilde\Omega$ is only assumed to be symmetric in $x_1$, estimate \eqref{grad_u_lower_bound} still holds when $\gamma$ is sufficiently large.
\end{remark}

\begin{remark}
It is noteworthy that Theorems \ref{thm_upperbound} and \ref{main_thm} are the first blow-up results in the context of field concentration for imperfect bonding interfaces. Note that 
$$\alpha \to \alpha_0:=\frac {-(n-1)+\sqrt{(n-1)^2+4(n-2)}}{2} \quad \mbox{as } \gamma\to \infty. $$ Therefore, the blow-up exponent in \eqref{grad_u_upperbound2} and \eqref{grad_u_lower_bound} converges to the optimal blow-up exponent $(\alpha_0-1)/2$ for the insulated conductivity problem with perfect bonding interfaces (see \cite{DLY}). This suggests that the perfect conducting inclusions with imperfect interfaces behave similarly to insulators when the bonding parameter $\gamma$ is large.
\end{remark}

\begin{remark}
When $D_1$ and $D_2$ are balls of radius $R$, centered at $(0,-R-\varepsilon/2)$ and $(0,R+\varepsilon/2)$ respectively, it holds that $\mu=1/R$, and therefore our critical value of bonding parameter for the dichotomy $\gamma=1/\mu=R$ as in Theorems \ref{thm_upperbound} and \ref{main_thm} is exactly the special bonding parameter ensuring that the inclusions are neutral to uniform fields, i.e., the insertion of the inclusions does not perturb the uniform fields (see \cites{HASHIN1992767,EVERY1992123,prl, LipVer,BenMil,fukushima2024finiteness}).  Indeed, in this critical case when $\gamma=R$, the linear potential $u=x_j$ ($j=1,2,\ldots,n$) automatically satisfies the homogeneous robin boundary condition on $\partial D_1\cup\partial D_2$ and thus is a solution to the equation \eqref{eq1} with $\varphi=x_j$. Our result reveals, for the first time, a fascinating microscopic dichotomy of the electric field at this critical bonding parameter $\gamma$.
\end{remark}

The rest of the paper is organized as follows. In the next section, we prove the uniform boundedness of $U_1$ and $U_2$, and the existence and uniqueness of solutions to Eq. \eqref{eq1}. In Section \ref{sec3}, we demonstrate the proofs of Theorem \ref{thm-1/2} and Corollary \ref{cor2} by utilizing a delicate change of variables, which preserves the normal directions on the upper and lower boundaries of $\Omega_{R_0/2}$, and an extension argument.  Theorem \ref{thm_polynomial_upperbound} is proved in Section \ref{sec_polybound}. Sections \ref{sec4}--\ref{sec6} are devoted to the proofs of Theorems \ref{thm_upperbound} and \ref{main_thm}, for which we use a dimension reduction argument to reduce the original problem into a degenerate elliptic equation in a spherical domain of $\bR^{n-1}$.  The Robin boundary condition on the interfaces introduces extra difficulties to the reduced problem in $\bR^{n-1}$ compared to the insulated case (Neumann boundary conditions), and eventually lead to a dichotomy for the field concentration depending on the bonding parameter $\gamma$.

\section{Preliminary}
In this section, we prove the uniform boundedness of $U_1$ and $U_2$ (Lemma \ref{lem:bdd:U}), and the existence and uniqueness of solutions to Eq. \eqref{eq1} (Lemma \ref{lemma_EU}).

\begin{lemma}\label{lem:bdd:U}
Let $n\ge 2$, $\Omega$, $D_1$, and $D_2$ be Lipschitz domains in $\bR^n$, $u\in H^1(\widetilde\Omega)$ be the solution to \eqref{eq1}, and $U_1$, $U_2$ be the same constants as in \eqref{eq1}. Then it holds that
\begin{equation}\label{wmp}
     \inf_{\partial \Omega} \varphi \le u \le \sup_{\partial \Omega} \varphi \quad \mbox{in } \widetilde\Omega,
\end{equation}
and
\begin{equation}\label{bdd:U}
    \inf_{\partial \Omega} \varphi \le U_j \le \sup_{\partial \Omega} \varphi, \quad j=1,2.
\end{equation}
\end{lemma}
\begin{proof}
By taking the average of both sides of the equality \eqref{eq1}$_2$ on $\partial D_j$ and using \eqref{eq1}$_3$, we obtain
\begin{equation}\label{averaging}
   U_j= (u)_{\partial D_j}, \ j=1,2.
\end{equation}
Thus it suffices to prove \eqref{wmp}, since \eqref{bdd:U} follows directly from \eqref{wmp} and \eqref{averaging}.
Without loss of generality, we only show that
\begin{equation*}
    u \le \sup_{\partial \Omega} \varphi \quad \mbox{in } \widetilde\Omega.
\end{equation*}
By considering $u-\sup_{\partial \Omega} \varphi$ in place of $u$, we may assume $u\le 0$ on $\partial\Omega$, and it suffices to show that
\begin{equation}\label{v_upper}
    u\le 0 \quad \mbox{in } \widetilde\Omega.
\end{equation}
Let $u_+:=\max\{u,0\}$. Testing the equation \eqref{eq1} with $u_+$ and using \eqref{averaging} yield 
\begin{align*}
    0&=\int_{\widetilde\Omega} |\nabla u_+|^{2}- \sum_{j=1}^2 \int_{\partial D_j} u_+\,\partial_\nu u=\int_{\widetilde\Omega} |\nabla u_+|^{2}+ \gamma^{-1}\sum_{j=1}^2 \int_{\partial D_j} u_+\,\big(u-(u)_{\partial D_j}\big)\\
&= \int_{\widetilde\Omega} |\nabla u_+|^{2}+ \gamma^{-1}\sum_{j=1}^2 \int_{\partial D_j} \big(u_+^2-(u)_{\partial D_j} (u_+)_{\partial D_j}\big)
\ge \int_{\widetilde\Omega} |\nabla u_+|^{2}.
\end{align*}
Here we used the fact $(u)_{\partial D_j}\le  (u_+)_{\partial D_j}$ and the Cauchy-Schwarz inequality in the last line. Therefore, since $u_+=0$ on $\partial \Omega$, we get $u_+\equiv 0$ in $\widetilde\Omega$, and thus \eqref{v_upper} holds. The proof is completed.
\end{proof}

\begin{lemma}\label{lemma_EU}
Let $n\ge 2$, $\Omega$, $D_1$, and $D_2$ be Lipschitz domains in $ \bR^n$.Then $u$ is the minimizer of \eqref{minimizer} if and only if $u\in H^1(\widetilde\Omega)$ satisfies \eqref{eq1}. Therefore, there exists a unique solution $u\in H^1(\widetilde\Omega)$ to Eq. \eqref{eq1}.
\end{lemma}
\begin{proof}
    First, we prove that \eqref{eq1} has at most one solution $u\in H^1(\widetilde\Omega)$. Let $u_1, u_2\in W^{1,p}(\Omega)$ be two solutions of \eqref{eq1}. Then by linearity, we know that $u_0:=u_1-u_2$ is a solution to Eq. \eqref{eq1} with $\varphi= 0$. By Lemma \ref{lem:bdd:U}, we know that $u_0\equiv 0$ in $\widetilde\Omega$, and thus $u_1 \equiv u_2$. It is straightforward to see that there exists a unique minimizer of \eqref{minimizer}, due to the convexity of $I[\cdot]$ and $\sA$. Therefore, it suffices to show that the minimizer $u$ of \eqref{minimizer} satisfies \eqref{eq1}. We show this by taking different test functions $v \in \sA$ in the equation
\begin{equation}\label{critical_point}
0 = \left.\frac{d}{dt} \right|_{t = 0}  I[u + t v].
\end{equation}
First we take $v \in C_c^\infty (\widetilde\Omega)$. Then \eqref{critical_point} reads as
$$
0 = \int_{\widetilde\Omega} \nabla u \cdot \nabla v.
$$
This implies
\begin{equation}\label{delta=0}
\Delta u =0  \quad\mbox{in }\widetilde{\Omega}.
\end{equation}
Next, we take $v \in H_0^1(\Omega)$. By using integration by parts and \eqref{delta=0}, from \eqref{critical_point} we deduce
\begin{equation}\label{E-L}
    \begin{aligned}
        0&=\int_{\widetilde\Omega} \nabla u \cdot \nabla v+ \gamma^{-1} \sum_{j=1}^2 \int_{\partial D_j} \big(u-(u)_{\partial D_j}\big)\big(v-(v)_{\partial D_j}\big)\\
        &=\int_{\partial D_1\cup \partial D_2} v\, \partial_\nu u+\gamma^{-1} \sum_{j=1}^2 \int_{\partial D_j} \big(u-(u)_{\partial D_j}\big)\big(v-(v)_{\partial D_j}\big)\\
        &=\sum_{j=1}^2 \int_{\partial D_j} v\,\Big(\partial_\nu u+\gamma^{-1} \big(u-(u)_{\partial D_j}\big)\Big).
    \end{aligned}
\end{equation}
Since $D_1$, $D_2$ are disjoint and \eqref{E-L} holds for any $v\in H^1_0(\Omega)$, we know that
\begin{align}\label{ro}
    u+\gamma\, \partial_\nu u = (u)_{\partial D_j} \quad \mbox{on } \partial D_j, \quad j=1,2.
\end{align}
Finally, we choose $v\in H_0^1(\Omega)$ such that $v = 1$ in $\overline{D_1}$ and $v = 0$ in $\overline{D_2}$. 
Then \eqref{E-L} implies 
\begin{equation}\label{no-flux-1}
    \int_{\partial D_1} \partial_\nu u=0.
\end{equation}
Similarly, by choosing $v\in H_0^1(\Omega)$ such that $v = 0$ in $\overline{D_1}$ and $v = 1$ in $\overline{D_2}$, one can also show that 
\begin{equation}\label{no-flux-2}
    \int_{\partial D_2} \partial_\nu u=0.
\end{equation}
By \eqref{delta=0} and \eqref{ro}--\eqref{no-flux-2}, we can conclude that the minimizer $u$ of \eqref{minimizer} is a solution to \eqref{eq1} with $U_j=(u)_{\partial D_j}$ ($j=1,2$).
The lemma is proved.
\end{proof}

\section{A flattening and extension argument}\label{sec3}
In this section, we use a flattening and extension argument to prove Theorem \ref{thm-1/2}.  Even though we are working on the Laplace equation with Robin boundary conditions, the argument we present here is very robust, and can be applied to more general settings.

Without loss of generality, we assume $R_0=1$. For general $R_0>0$, we can make the scaling $z=x/R_0$, and consider Eq. \eqref{main_problem_narrow} in the $z$-coordinates. We also assume $U_2=-U_1$, since otherwise we can consider the equation for $u-(U_1+U_2)/2$ instead of $u$. Throughout this section, unless specify otherwise, we use $C$ to denote positive constants that could be different from line to line, but depend only on $n$, $\gamma$, $c_1$, and  $c_2$, where $c_1$ and $c_2$ are defined in \eqref{fg_1} and \eqref{def:c_2}, respectively. We fix $x=x_0\in \Omega_{1/2}$, and prove Theorem \ref{thm-1/2} at $x=x_0$. Recall that $$r=\frac{1}{4}(\varepsilon+|x_0'|^2)^{1/2}.$$ By the triangle inequality, for any $x\in \overline{\Omega_{x_0,2r}}$, we have
\begin{equation}\label{x':est}
    \frac{1}{4}|x_0'|^2-\frac{1}{4}\varepsilon\le \frac{1}{2} |x_0'|^2-|x'-x_0'|^2\le |x'|^2\le 2|x'-x_0'|^2+2|x_0'|^2\le \frac{5}{2}|x_0'|^2+\frac{1}{2} \varepsilon.
\end{equation}
By \eqref{fg_0}--\eqref{def:c_2}, we know that
\begin{equation}\label{height_bdd}
    \frac{1}{C} \,(\varepsilon+|x'|^2)\le \varepsilon+f(x')-g(x')\le C \,(\varepsilon+|x'|^2).
\end{equation}
Combining \eqref{x':est} and \eqref{height_bdd}, we obtain
\begin{equation}\label{height_final}
    \frac{1}{C}\, r^2\le \varepsilon+f(x')-g(x')\le C\,r^2
\end{equation}
for any $x\in \overline{\Omega_{x_0,2r}}$.

Let $r_0=r_0(n,c_1,c_2) \in (0, 1/4)$ be an $\varepsilon$-independent constant to be determined later. Then by the classical gradient estimates (c.f. \cite{MR3059278}*{Chapter 4}), \eqref{gradient-1/2} holds at $x=x_0$ whenever $r=\frac{1}{4}(\varepsilon+|x_0'|^2)^{1/2}>r_0$. Therefore, it suffices to prove Theorem \ref{thm-1/2} for the case when $r\in(0,r_0]$.\\

\textbf{Step 1: The flattening map.}

To ensure that the solution after the change of variables still satisfies Robin boundary conditions, we adapt the flattening map introduced in \cite{DYZ23}*{Section 3} to our setting, which preserves the normal directions on the upper and lower boundaries.

More precisely, we define
$$
\tilde{f}(x') :=
\left\{
\begin{aligned}
&f(x') && \mbox{when}~~|x'| \le 10\,r_0,\\
&0 &&  \mbox{when}~~|x'| > 10\,r_0,
\end{aligned}
\right.
$$
and 
$$
\tilde{g}(x') :=
\left\{
\begin{aligned}
&g(x') && \mbox{when}~~|x'| \le 10\,r_0,\\
&0 &&  \mbox{when}~~|x'| > 10\,r_0.
\end{aligned}
\right.
$$
We denote
$$
Q_{s,t} := \{ y = (y', y_n) \in \bR^n:~ |y'-x_0'| < s, |y_n| < t \},
$$
and for $y \in \overline{Q_{2r,r^2}}$, we define the map $x = \Phi(y)$ by
\begin{equation}\label{Phi}
\left\{
\begin{aligned}
x' &= y' - h(y),\\
x_n &= \frac{1}{2} \Big[\frac{y_n}{r^2} (\varepsilon + \tilde f(y') - \tilde g(y')) + \tilde f(y') +\tilde g(y') \Big],
\end{aligned}
\right.
\end{equation}
where
\begin{equation*}
h(y) = (y_n - r^2)(y_n + r^2)(\Theta y_n + \Xi),
\end{equation*}
$$
\left\{
\begin{aligned}
\Theta &= \frac{1}{8r^6} [\varepsilon + \tilde f(y') - \tilde g(y')] D_{y'} [\tilde f^\kappa(y') + \tilde g^\kappa(y')],\\
\Xi &= \frac{1}{8r^4} [\varepsilon + \tilde f(y') - \tilde g(y')] D_{y'} [\tilde f^\kappa(y') - \tilde g^\kappa(y')],
\end{aligned}
\right.
$$
$\tilde f^\kappa$ and $\tilde g^\kappa$ are mollifications of $\tilde f$ and $\tilde g$ given by
\begin{equation}\label{mollification}
\tilde f^\kappa (y') := \int_{\bR^{n-1}} \tilde f(y' - \kappa z') \varphi(z') \, dz',\quad \tilde g^\kappa (y') := \int_{\bR^{n-1}} \tilde g(y' - \kappa z') \varphi(z') \, dz',
\end{equation}
where $\varphi$ is a positive smooth function with unit integral supported in $B_1 \subset \bR^{n-1}$,
and 
$$\kappa = \frac{r^4 - y_n^2}{r}\ge 0.$$
As in \cite{DYZ23}*{Section 3}, the mollification \eqref{mollification} was introduced to overcome the lack of regularities of $f$ and $g$. Note that for any $y \in \overline{Q_{2r,r^2}}$, by the triangle inequality, we have 
$$|y'|\le |x_0'|+2r\le 6r\le 6r_0,$$
and thus 
\begin{equation*} 
    \tilde f(y')=f(y') \quad \mbox{and} \quad  \tilde{g}(y')=g(y').
\end{equation*}
Thus by \eqref{fg_0}, for $y \in \overline{Q_{2r,r^2}}$, we have
\begin{equation}\label{h_derivatives}
\begin{aligned}
&|D_{y'}^k \tilde f(y') | \le Cr^{2-k}, \quad 
|D_{y'}^k \tilde g(y') | \le Cr^{2-k}, \quad k=0,1,2,\\
&|D_{y'}^k \tilde f^\kappa(y') | \le Cr^{2-k}, \quad 
|D_{y'}^k \tilde g^\kappa(y') | \le Cr^{2-k}, \quad k=0,1,2.
\end{aligned}
\end{equation}

Similar to \cite{DYZ23}*{Lemma 3.1}, one can obtain several properties for the change of variables $\Phi$ as in the lemma below. The proof will be given in the Appendix for completeness. 
\begin{lemma}\label{Change_of_variable_lemma}
There exists an $r_0=r_0(n,c_1,c_2)\in(0,1/4)$ independent of $\varepsilon$  and $\gamma$, such that when $r \in (0, r_0]$ and $\Phi$ is given as \eqref{Phi}, we have
\begin{enumerate}
\item There exists a positive constant $C$ independent of $\varepsilon$ and $r$, such that
$$\frac{I}{C} \le D \Phi(y) \le C I, \quad y \in \overline{Q_{2r,r^2}},$$
and hence $\Phi$ is invertible.
\item $$Q_{1.9r,r^2} \subset \Phi^{-1}(\Omega_{ x_0, 2r})\quad
\mbox{and} \quad
\Omega_{x_0, r} \subset \Phi (Q_{1.1r, r^2}).$$
\item Let $u \in H^{1}(\Omega_{x_0,2r})$ be a solution of
\begin{equation}\label{local_equation}
\left\{
\begin{aligned}
\Delta u &=0 &&\mbox{in }\Omega_{x_0,2r},\\
u+\gamma \partial_\nu u &= U_1 &&\mbox{on } \Gamma_+  \cap \overline{\Omega_{x_0,2r}},\\
u+\gamma \partial_\nu u &= -U_1 &&\mbox{on } \Gamma_- \cap \overline{\Omega_{x_0,2r}},\\
\end{aligned}
\right.
\end{equation}
and $\tilde u(y) = u(\Phi(y))$. Then $\tilde u$ is a solution of the following elliptic equation
\begin{equation}\label{tilde_u_equation}
\left\{
\begin{aligned}
a^{ij}D_{y_iy_j}\tilde u+b^i D_{y_i}\tilde u&=0  &&\mbox{in }Q_{1.9r, r^2},\\
\tilde u+\gamma \,F_1^{-1} D_{y_n}\tilde u&=U_1 &&\mbox{on } \{y_n=r^2\}\cap \overline{Q_{1.9r, r^2}},\\
\tilde u-\gamma \, F_2^{-1} D_{y_n}\tilde u&=-U_1 &&\mbox{on } \{y_n=-r^2\}\cap \overline{Q_{1.9r, r^2}},
\end{aligned}
\right.
\end{equation}
where 
$a:=(a^{ij})_{n\times n}=(D\Phi)^{-1}((D\Phi)^{-1})^T\in C^{0,1}(\overline{Q_{1.9r, r^2}}:\bR^{n\times n})$, $b:=(b^i)_{n} \in L^\infty(\overline{Q_{1.9r, r^2}}:\bR^{n})$ satisfy
\begin{equation}\label{est:a}
\frac{I}{C} \le a \le C I, \quad |D_y a| \le \frac{C}{r}, \quad |b| \le \frac{C}{r},
\end{equation}
and
\begin{equation}\label{ortho}
a_{nj}=a_{jn}=0 \quad \mbox{on } \{y_n=\pm r^2\}\cap \overline{Q_{1.9r, r^2}},
\end{equation}
for any  $j\in \{1,2,\ldots, n-1\}$,
and $F_i\in C^{0,1}(\overline{Q_{1.9r, r^2}})\cap C^{1,1}(Q_{1.9r, r^2})$ ($i=1,2$) are defined as
\begin{equation*}
\left\{
\begin{aligned}
&F_1:=F_1(y):=\frac{1}{2r^2}(\varepsilon+f(y')-g(y'))\sqrt{1+|D_{y'} \tilde f^\kappa(y')|^2},\\
&F_2:=F_2(y):=\frac{1}{2r^2}(\varepsilon+f(y')-g(y'))\sqrt{1+|D_{y'} \tilde g^\kappa(y')|^2},
\end{aligned}
\right.
\end{equation*}
satisfying
\begin{equation}\label{F_bound}
\left\{
\begin{aligned}
&\frac{1}{C}\le F_i \le C, \quad |D_{y'} F_i|\le \frac{C}{r}, \quad |D_{y_n} F_i|\le C\, r^2,\\
& |D^2_{y'}F_i(y)|\le \frac{C\,r^2}{r^4-y_n^2},\quad |D_{y_n}D_{y'}F_i(y)|\le \frac{C\,r^3}{r^4-y_n^2}, \quad 
|D^2_{y_n}F_i(y)|\le \frac{C\,r^4}{r^4-y_n^2}.
\end{aligned}
\right.\\
\end{equation}
\end{enumerate}
\end{lemma}

\textbf{Step 2: Reducing to the homogeneous Robin boundary condition.}

Next, we construct an auxiliary function $w=w(y)$, such that $v:=\tilde u -w$ satisfies the homogeneous Robin boundary conditions and $w$ has good estimates. Ideally, we would like to take $w(y)$ to be the form $a(y')y_n + b(y')$. However, due to the lack of regularities of $f$ and $g$, the second derivative of $a(y')$ is very singular if we adapt the same mollification \eqref{mollification} to $f$ and $g$. To overcome this, we split $w$ into $w_1+w_2$, where $w_1$ is a linear function in $y_n$ whose coefficient is $C^{1,1}$ and approximates $a(y')$ well enough, and $w_2$ is ``small", as follows.

Let $F_0:=\frac{1}{2r^2}(\varepsilon+f-g)$. First, we choose $$w_1(y):=A\,y_n,$$ where 
\begin{equation}\label{def:A}
 A:=A(y'):=\frac{U_1}{r^2+\gamma\, F_0^{-1}}.
\end{equation}
Then $w_1$ satisfies
\begin{equation}\label{eq:w1}
\left\{
\begin{aligned}
&w_1+\gamma \,F_0^{-1}\, D_{y_n}w_1= U_1 
&&\mbox{on } \{y_n=r^2\}\cap \overline{Q_{1.9r, r^2}},\\
&w_1-\gamma \, F_0^{-1}\, D_{y_n}w_1= -U_1 &&\mbox{on } \{y_n=-r^2\}\cap \overline{Q_{1.9r, r^2}}.
\end{aligned}
\right.
\end{equation}
Now we choose $w_2$ such that
\begin{equation}\label{eq:w2}
\left\{
\begin{aligned}
&w_2+\gamma \,F_1^{-1}\, D_{y_n}w_2= \gamma (F_0^{-1}-F_1^{-1})D_{y_n}w_1\quad\mbox{on } \{y_n=r^2\}\cap \overline{Q_{1.9r, r^2}},\\
&w_2-\gamma \,F_2^{-1}\, D_{y_n}w_2= \gamma (F_2^{-1}-F_0^{-1})D_{y_n}w_1\quad \mbox{on } \{y_n=-r^2\}\cap \overline{Q_{1.9r, r^2}}.
\end{aligned}
\right.
\end{equation}
We use the ansatz $$w_2=(y_n+r^2)(y_n-r^2)(\overline{A}\,y_n+\overline{B}),$$ where 
$$ \overline{A}:=\overline{A}(y):=\frac{A}{4r^4}\big( F_1(F_0^{-1}-F_1^{-1})-F_2(F_2^{-1}-F_0^{-1})\big)=\frac{A}{4r^4F_0}\big((F_1-F_0)-(F_0-F_2)\big),$$
$$ \overline{B}:=\overline{B}(y):=\frac{A}{4r^2}\big( F_1(F_0^{-1}-F_1^{-1})+F_2(F_2^{-1}-F_0^{-1})\big)=\frac{A}{4r^2F_0}\big((F_1-F_0)+(F_0-F_2)\big),$$
and $A=A(y')$ is the same function as in \eqref{def:A}, so that
$w_2$ satisfies \eqref{eq:w2}. Note that $w_2$ is sufficiently small since $F_i$ is close to $F_0$ (see the proof of Lemma \ref{lem:minus} below). Finally, we let $$w:=w_1+w_2.$$ Then by \eqref{eq:w1} and \eqref{eq:w2}, we know that $w$ satisfies 
\begin{equation}\label{eq:w}
\left\{
\begin{aligned}
&w+\gamma \,F_1^{-1}\, D_{y_n}w= U_1 &&\mbox{on } \{y_n=r^2\}\cap \overline{Q_{1.9r, r^2}},\\
&w-\gamma \, F_2^{-1}\, D_{y_n}w= -U_1 &&\mbox{on } \{y_n=-r^2\}\cap \overline{Q_{1.9r, r^2}}.
\end{aligned}
\right.
\end{equation}
Let $v=\tilde{u}-w$. Then we have the following lemma.

\begin{lemma}\label{lem:minus}
Let $\tilde u$, $a$, $b$, $F$, $G$ be defined as in Lemma \ref{Change_of_variable_lemma}, and $w$, $v$ be as above. Then
$v$ is a solution to the following equation 
\begin{equation}\label{v_equation}
\left\{
\begin{aligned}
a^{ij}D_{y_iy_j} v+b^i D_{y_i} v&=- a^{ij}D_{y_iy_j} w-b^i D_{y_i} w&&\mbox{in }Q_{1.9r, r^2},\\
v+\gamma \,F_1^{-1} D_{y_n}v&=0 &&\mbox{on } \{y_n=r^2\}\cap \overline{Q_{1.9r, r^2}},\\
v-\gamma \,F_2^{-1} D_{y_n}v&=0 &&\mbox{on } \{y_n=-r^2\}\cap \overline{Q_{1.9r, r^2}}.
\end{aligned}
\right.
\end{equation}
Moreover, we have $w\in C^{1,1}(\overline{Q_{1.9r, r^2}})$, satisfying
\begin{equation}\label{est:w}
|w|\le C\,r^2\,|U_1|,\quad
|D_y w|\le C\,|U_1|,\quad |D^2_{y} w|\le \frac{C\,|U_1|}{r}.
\end{equation}
\end{lemma}
\begin{proof}
    By direct computations, using \eqref{fg_0}--\eqref{def:c_2}, we obtain
    \begin{equation*}
       |F_0|\le C,\quad |D_{y'}F_0|\le \frac{C}{r}, \quad |D^2_{y'}F_0|\le \frac{C}{r^2}, 
    \end{equation*}
and 
\begin{equation}\label{est:oF-1}
    |F_0^{-1}|\le C,\quad |D_{y'}F_0^{-1}|\le \frac{C}{r}, \quad |D^2_{y'}F_0^{-1}|\le \frac{C}{r^2}.
\end{equation}
Thus, $A=A(y')$ defined in \eqref{def:A} satisfies
\begin{equation}\label{est:A}
       |A|\le C\,|U_1|,\quad |D_{y'}A|\le \frac{C\,|U_1|}{r}, \quad |D^2_{y'}A|\le \frac{C\,|U_1|}{r^2}.
\end{equation}
Thus, by noting that $D_{y_n}^2 w_1=0$, we have
\begin{equation}\label{est:w1}
    |w_1|\le C\,r^2\,|U_1|,\quad |D_{y}w_1|\le C\,|U_1|, \quad |D^2_{y}w_1|\le \frac{C\,|U_1|}{r}.
\end{equation}
Note that for any $t\in\mathbb{R}$, it holds that
$$0\le \sqrt{1+t^2}-1\le \frac{t^2}{2}.$$
Therefore, by \eqref{fg_0}--\eqref{def:c_2}, and \eqref{F_bound}, we have
\begin{equation}\label{close_F}
    |F_i-F_0|\le C \,r^2, \quad i=1,2.
\end{equation}
Moreover, by using \eqref{fg_0}--\eqref{def:c_2}, and \eqref{h_derivatives}, we also have
\begin{equation}\label{close_DF}
    |D_{y'}F_i-D_{y'}F_0|\le C \,r, \quad i=1,2.
\end{equation}
By direct computations and using \eqref{est:oF-1}, \eqref{est:A}, \eqref{close_F}, and \eqref{close_DF}, we have the following estimates for $\overline{A}$:
\begin{equation}\label{est:oA}
    |\overline{A}|\le C\, r^{-2}\,|U_1|,\quad |D_{y'}\overline{A}|\le C\,r^{-3}\,|U_1|,\quad |D_{y_n}\overline{A}|\le C\,r^{-2}\,|U_1|,
\end{equation}
\begin{equation}\label{est:oA2}
    |D^2_{y'}\overline{A}|\le \frac{C\,|U_1|}{r^2(r^4-y_n^2)},\quad |D_{y_n}D_{y'}\overline{A}|\le \frac{C\,|U_1|}{r (r^4-y_n^2)}, \quad |D^2_{y_n}\overline{A}|\le \frac{C\,|U_1|}{r^4-y_n^2}.
\end{equation}
Similarly, we have
\begin{equation}\label{est:oB}
    |\overline{B}|\le C\, |U_1|,\quad |D_{y'}\overline{B}|\le C\,r^{-1}\,|U_1|,\quad |D_{y_n}\overline{B}|\le C\,|U_1|,
\end{equation}
\begin{equation}\label{est:oB2}
    |D^2_{y'}\overline{B}|\le \frac{C\,|U_1|}{r^4-y_n^2},\quad |D_{y_n}D_{y'}\overline{B}|\le \frac{C\,r\,|U_1|}{r^4-y_n^2}, \quad |D^2_{y_n}\overline{B}|\le \frac{C\,r^2\,|U_1|}{r^4-y_n^2}.
\end{equation}
By \eqref{est:oA}--\eqref{est:oB2} and direct computations, one can deduce that
\begin{equation}\label{est:w2}
    |w_2|\le C\,r^4\,|U_1|,\quad |D_{y}w_2|\le C\,r^2\,|U_1|, \quad |D^2_{y}w_2|\le C\,|U_1|.
\end{equation}
By \eqref{est:w1} and \eqref{est:w2}, we have $w=w_1+w_2\in C^{1,1}(\overline{Q_{1.9r, r^2}})$, satisfying \eqref{est:w}. Thus by \eqref{tilde_u_equation} and \eqref{eq:w}, $v$ is a solution to \eqref{v_equation}.
\end{proof}

 \textbf{Step 3: Reducing to the homogeneous Neumann boundary condition.}

We make another transformation $\tilde v= v \, e^{\psi}$, where $\psi=\psi(y)$ is a function to be determined, such that $\tilde v$ satisfies the homogeneous Neumann boundary conditions. 
Note that by the chain rule, 
$$D_{y_n}\tilde v=e^{\psi}(D_{y_n}v+v\,D_{y_n} \psi).$$
Thus if we choose $$\psi=(y_n+r^2)(y_n-r^2)(\tilde{A}\,y_n+\tilde{B}),$$ where $\tilde A$ and $\tilde B$ are functions such that
\begin{equation*}
\left\{
\begin{aligned}
&D_{y_n}\psi(y',r^2)=2r^2(\tilde A r^2+\tilde B)=\gamma^{-1} F_1,\\
&D_{y_n}\psi(y',-r^2)=-2r^2(-\tilde A r^2+\tilde B)=-\gamma^{-1} F_2,
\end{aligned}
\right.
\end{equation*}
namely, 
$$\tilde A:=(4\gamma r^4)^{-1}(F_1-F_2) \quad \mbox{and}\quad \tilde B:=(4\gamma r^2)^{-1}(F_1+F_2).$$
then by \eqref{v_equation}$, \tilde v:= v \, e^{\psi}:=v\, e^{(y_n+r^2)(y_n-r^2)(\tilde{A}\,y_n+\tilde{B})}$ satisfies 
\begin{equation}\label{v_boundary}
D_{y_n}\tilde v=0\quad\mbox{on } \{y_n=\pm r^2\}\cap \overline{Q_{1.9r, r^2}}.
\end{equation}
We have the following lemma.
\begin{lemma}
    Let $\psi$, $\tilde v$ be defined as above and $a$ be defined as in Lemma \ref{Change_of_variable_lemma}. Then $\tilde v$ is a solution to
\begin{equation}\label{v_tilde_equation}
\left\{
\begin{aligned}
{a}^{ij}D_{y_iy_j} \tilde v+\tilde{b}^i D_{y_i} \tilde v+\tilde c \,\tilde v&= \tilde h&&\mbox{in }Q_{1.9r, r^2},\\
D_{y_n}\tilde v&=0 &&\mbox{on } \{y_n=\pm r^2\}\cap \overline{Q_{1.9r, r^2}},
\end{aligned}
\right.
\end{equation}
where $\tilde b:=(\tilde b^i)_{n\times 1}\in L^\infty(\overline{Q_{1.9r, r^2}}:\bR^{n})$, and $ \tilde c,\, \tilde h\in L^\infty(\overline{Q_{1.9r, r^2}})$ satisfy
\begin{equation}\label{est:coeff}
    |\tilde b|\le \frac{C}{r},\quad |\tilde c|\le \frac{C}{r^2},\quad |\tilde h|\le \frac{C\,|U_1|}{r}.
\end{equation} 
\end{lemma}
\begin{proof}
    By \eqref{F_bound} and direct computations, we have $\psi\in C^{1,1}(\overline{Q_{1.9r, r^2}})$, satisfying
    \begin{equation}\label{est:psi}
     |\psi|\le C \,r^2, \quad |D_y\psi|\le C, \quad |D^2_{y} \psi|\le C \, r^{-2}. 
    \end{equation}
For $i,j\in\{1,\ldots,n\}$, we define
\begin{align*}
    & \tilde b^i=b^i-2\sum_{j=1}^n a^{ij}D_{y_j}\psi, \quad \tilde c=a^{ij}(D_{y_i}\psi \,D_{y_j}\psi-D_{y_i y_j}\psi)-b^iD_{y_i}\psi, \\ 
    &\tilde h=(- a^{ij}D_{y_iy_j} w-b^i D_{y_i} w)\,e^\psi,
\end{align*}
where $b^{i}$ and $w$ are the same as in Lemmas \ref{Change_of_variable_lemma} and \ref{lem:minus}, respectively.
Then \eqref{est:coeff} follows \eqref{est:a}, \eqref{est:w}, and \eqref{est:psi}. Thus by \eqref{v_equation}, \eqref{v_boundary}, and the chain rule, $\tilde v$ is a solution to \eqref{v_tilde_equation}.
\end{proof}

\textbf{Step 4: Extension and estimates.}

Next we extend $\tilde v$, $a$, $\tilde b$, $\tilde c$, and $\tilde h$ to the larger cylinder $\mathcal{Q}_{1.9r}:=\{(y',y_n)\in \mathbb{R}^n: \,|y'-x_0'|<1.9r\}$.
For $i,j = 1,2,\ldots, n-1$, we take the even extension of $ a^{ij}$, $ a^{nn}$, $\tilde b^i$, $\tilde c$, $\tilde h$, and $\tilde v$ with respect to $\{y_n=r^2\}$, and take the odd extension of $ a^{in}$, $ a^{ni}$, and $b^n$ with respect to $\{y_n=r^2\}$. Then we take the periodic extension of these functions in the $y_n$ axis, so that the period is equal to $4r^2$. We still denote these functions by $\tilde v$, $a$, $\tilde b$, $\tilde c$, and $\tilde h$ after the extension. Then because of the Neumann boundary condition, $\tilde v$ is a solution of the following equation
\begin{equation}\label{eq:final}
a^{ij}D_{y_iy_j} \tilde v+\tilde{b}^i D_{y_i} \tilde v+\tilde c \,\tilde v= \tilde h \quad \mbox{in }\mathcal{Q}_{1.9r}.
\end{equation}
Note that because of \eqref{ortho}, after the extension, we have $a\in C^{0,1}(\overline{\mathcal{Q}_{1.9r}}:\bR^{n\times n})$ satisfies
\begin{equation}\label{lip-coeff}
      \frac{I}{C} \le a \le C I,\quad  |D_y a|\le C.
\end{equation}
Thus we have the following gradient estimate for $\tilde v$.
\begin{lemma}
    Let $\tilde v$ be a solution of \eqref{eq:final}. Then it holds that
    \begin{equation}\label{estimate:tildev}
     \|D_y \tilde v\|_{L^\infty(Q_{1.1r, 1.1r})}+r^{-1}\|\tilde v\|_{L^\infty(Q_{1.1r, 1.1r})}\le C\big(r^{-1}\dashnorm{\tilde v}{L^2(Q_{1.9r, 1.9r})}+|U_1|\big).
    \end{equation}
\end{lemma}
\begin{proof}
Observe that after the extensions defined as above, $\tilde b$, $\tilde c$ and $\tilde H$ are bounded in $\mathcal{Q}_{1.9r}$, still satisfying
\eqref{est:coeff}.
We consider the following rescaling of Eq. \eqref{eq:final}:
$$
\begin{aligned}
&\underline{v}(y)=\tilde v(x_0'+ry',ry_n), \quad \underline{a}(y)=a(x_0'+ry',ry_n), \quad \underline{b}(y)=r\,\tilde b(x_0'+ry',ry_n),\\
&\underline{c}(y)=r^2\,\tilde c(x_0'+ry',ry_n), \quad \underline{h}(y)=r^2\,\tilde h(x_0'+ry',ry_n).
\end{aligned}
$$
Then $\underline{v}$ is a solution of the following equation
$$
\underline{a}^{ij}D_{y_iy_j} \underline v+\underline{b}^i D_{y_i} \underline v+\underline c \,\underline v= \underline h \quad \mbox{in }\mathcal{C}_{1.9}.
$$
where $\mathcal{C}_{1.9}:=\{y=(y',y_n)\in\mathbb{R}^n:\, |y'|<1.9\}$.
By  \eqref{est:coeff} and \eqref{lip-coeff}, we know that $\underline a$ is Lipschitz continuous in $\mathcal{C}_{1,9}$, and  $\underline b$, $\underline c$, $\underline h$ are bounded in $\mathcal{C}_{1,9}$,  satisfying
\begin{equation*}
   \frac{I}{C} \le \underline{a} \le C I,\quad  |D_y\underline a|+|\underline b|+|\underline c|\le C, \quad |\underline h|\le C\,r\,|U_1|.
\end{equation*}
Then by the classical $W^{2,p}$ estimate (c.f. \cite{GT}*{Theorem 9.11}) and a standard iteration argument, for any $p>1$, it holds that 
\begin{equation}\label{schauder1}
    \| \underline v\|_{W^{2,p}(\mathcal{Q}_{1.1})}\le C_1(p)\big(\|\underline v\|_{L^p(\mathcal{Q}_{1.5})}+\|\underline h\|_{L^{p}(\mathcal{Q}_{1.5})}\big)\le  C_2(p)\big(\|\underline v\|_{L^2(\mathcal{Q}_{1.9})}+\|\underline h\|_{L^{p}(\mathcal{Q}_{1.9})}\big).
\end{equation}
Here $C_1(p)$, $C_2(p)$ are two constants depending only on $n$, $c_1$, $c_2$, $\gamma$, and $p$.
We fix some $p>n$ and take $\theta=1-n/p$. Then by using the Sobolev embedding
$W^{2,p}(\mathcal{Q}_{1.9})\subset C^{1,\theta}(\mathcal{Q}_{1.9})$, the estimate \eqref{schauder1} also implies 
\begin{equation}\label{schauder2}
    \|\underline v\|_{L^{\infty}(\mathcal{Q}_{1.1})}+ \|D_y\underline v\|_{L^{\infty}(\mathcal{Q}_{1.1})}\le  C\big(\|\underline v\|_{L^2(\mathcal{Q}_{1.9})}+r\,|U_1|\big).
\end{equation}
Scaling back \eqref{schauder2} gives \eqref{estimate:tildev}. The lemma is proved.
\end{proof}

Now we are ready to prove Theorem \ref{thm-1/2}.
\begin{proof}[Proof of Theorem \ref{thm-1/2}.]
As is explained at the beginning of this section, we prove the theorem at $x=x_0$ and it suffices to prove \eqref{gradient-1/2} for the case when $R_0=1$, $U_2=-U_1$, and $r:=\frac{1}{4}(\varepsilon+|x_0'|^2)^{1/2}\le r_0$, where $r_0=r_0(n,c_1,c_2)$ is given in Lemma \ref{Change_of_variable_lemma}.
Since $u$ is a solution to \eqref{local_equation}, we define $\tilde u$, $w$, $v$, and $\tilde v$ as earlier in this section.
Since $\tilde v$ is even with respect to $\{y_n=r^2\}$ and periodic in $y_n$ with a period of $4r^2$, we know that \eqref{estimate:tildev} directly implies 
$$
\|D_y \tilde v\|_{L^\infty(Q_{1.1r, r^2})}+r^{-1}\|\tilde v\|_{L^\infty(Q_{1.1r, r^2})}\le C\big(r^{-1}\dashnorm{\tilde v}{L^2(Q_{1.9r, r^2})}+|U_1|\big).
$$   
Since $v=\tilde v e^{-\psi}$, the last inequality and \eqref{est:psi} also imply
\begin{equation}\label{est:v}
\begin{aligned}
&\|D_y v\|_{L^\infty(Q_{1.1r, r^2})}\le \|\tilde v \,D_ye^{-\psi}\|_{L^\infty(Q_{1.1r, r^2})}
+ \|D_y\tilde v \,e^{-\psi}\|_{L^\infty(Q_{1.1r, r^2})}\\
&\le C \|\tilde v\|_{L^\infty(Q_{1.1r, r^2})}+ 
C\big(r^{-1}\dashnorm{\tilde v}{L^2(Q_{1.9r, r^2})}+|U_1|\big)\\
&\le 
C\big(r^{-1}\dashnorm{v}{L^2(Q_{1.9r, r^2})}+|U_1|\big).
\end{aligned}
\end{equation}
Next, since $\tilde u=v+w$, from \eqref{est:v} we deduce that
\begin{equation}\label{est:tildeu}
\begin{aligned}
    &\|D_y \tilde u\|_{L^\infty(Q_{1.1r, r^2})}\le \| D_y v\|_{L^\infty(Q_{1.1r, r^2})}
+ \|D_y w\|_{L^\infty(Q_{1.1r, r^2})}\\
&\le 
C\big(r^{-1}\dashnorm{v}{L^2(Q_{1.9r, r^2})}+|U_1|\big)+\|D_y w\|_{L^\infty(Q_{1.1r, r^2})}\\
&\le C\big(r^{-1}\dashnorm{\tilde u}{L^2(Q_{1.9r, r^2})}+r^{-1} \|w\|_{L^\infty(Q_{1.9r, r^2})}+|U_1|\big)+\|D_y w\|_{L^\infty(Q_{1.1r, r^2})}\\
&\le C\big(r^{-1}\dashnorm{\tilde u}{L^2(Q_{1.9r, r^2})}+|U_1|\big).
\end{aligned}   
\end{equation}
Here we used \eqref{est:w} in the last line.

Finally, since $\tilde u(y)=u(\Phi(y))$ in $\overline{Q_{1.9r,r^2}}$, by using Lemma \ref{Change_of_variable_lemma} (a) and (b), from \eqref{est:tildeu} we obtain \eqref{gradient-1/2}. The proof is completed.
\end{proof}
Next, we give the proof of Corollary \ref{cor2}.
\begin{proof}[Proof of Corollary \ref{cor2}.]
By classical gradient estimates (c.f. \cite{MR3059278}*{Chapter 4}), it suffices to prove \eqref{gradient-bdd} for the case when $\varepsilon \in(0,R_0/4)$ and $x\in \Omega_{R_0/4}$. By symmetry, if $u$ is a solution of \eqref{eq1}, then $\hat{u}(x):=-u(x',-x_n)$ satisfies
\begin{equation*}
\begin{cases}
\Delta \hat u =0&\mbox{in}~\widetilde\Omega,\\
\hat u + \gamma \partial_\nu \hat u = -U_2 &\mbox{on}~\partial D_1,\quad
\hat u + \gamma \partial_\nu \hat u = -U_1 \quad \mbox{on}~\partial D_2,\\
\int_{\partial D_j} \partial_\nu \hat u \, d\sigma = 0& j=1,2,\\
\hat u=\varphi(x)&\mbox{on}~\partial\Omega.
\end{cases}
\end{equation*}
Then by the uniqueness of solutions (Lemma \ref{lemma_EU}), we know that $u\equiv \hat u$. In particular, we have 
\begin{equation}\label{U1+U2}
    U_1=-U_2
\end{equation}
and 
\begin{equation}\label{0@0}
    u(x',0)= 0 \quad \mbox{when } (x',0)\in \overline{\widetilde\Omega}.
\end{equation}
By Theorem \ref{thm-1/2}, \eqref{U1+U2}, and \eqref{bdd:U}, for any $x\in \Omega_{R_0/2}$ and $r=\frac{1}{4}(\varepsilon+|x'|^2)^{\frac{1}{2}}$, it holds that
\begin{equation}\label{gradient-1/2_1}
|Du(x)| \le  C  \big(r^{-1} \dashnorm{u}{L^2(\Omega_{x,r})}  + |U_1|\big)\le  C  \big(r^{-1} \|u\|_{L^\infty(\Omega_{x,r})}  + \|\varphi\|_{L^\infty(\partial \Omega)}\big).
\end{equation}
Using \eqref{u_C1_outside}, \eqref{gradient-1/2_1} directly implies 
\begin{equation}\label{gradient-1/2_2}
|Du(x)| \le  C \,r^{-1} \|\varphi\|_{L^\infty(\partial \Omega)}.
\end{equation}
By the mean value theorem and \eqref{0@0}, for any $x=(x', x_n)\in \Omega_{R_0/2}$, there exists $\xi\in\mathbb{R}$ between $0$ and $x_n$, such that
$$
u(x)=u(x)-u(x',0)=D_n u(x',\xi)\,x_n.
$$
The last equality together with \eqref{gradient-1/2_2}, \eqref{fg_0}, and \eqref{def:c_2} yields 
$$|u(x)|\le C \,r\,\|\varphi\|_{L^\infty(\partial \Omega)}$$
for any $x\in \Omega_{R_0/2}$.
Combining \eqref{gradient-1/2_1}, \eqref{height_final}, and the last inequality yields \eqref{gradient-bdd}. The proof is completed.
\end{proof}

\section{Proof of Theorem \ref{thm_polynomial_upperbound}}\label{sec_polybound}

In this section, we prove Theorem \ref{thm_polynomial_upperbound}.

\begin{proof}[Proof of Theorem \ref{thm_polynomial_upperbound}]
As in the proof of Theorem \ref{thm-1/2}, without loss of generality, we may assume that $U_1 = - U_2$. It suffices to prove the theorem with $\beta > 0$.
Let $s_0 \in (0,R_0]$ be an $\varepsilon$-independent constant to be determined throughout the proof. We only need to prove \eqref{grad_u_polyupperbound} for $x \in \Omega_{s_0}$. For $s \in [\sqrt{\varepsilon}, s_0)$, we denote 
$$\Gamma_s:= (\Gamma_+ \cup \Gamma_-) \cap \overline{\Omega}_s.$$
Define an auxiliary function 
$$
v(x) := \frac{U_1}{\varepsilon/2 + \gamma}x_n, \quad x \in \Omega_{2s_0}
$$
and $\tilde{u} := u - v$. Then $\Delta \tilde{u} = 0$ in $\Omega_{2s_0}$. On $\Gamma_+$, we have
\begin{align*}
\tilde{u} + \gamma \partial_\nu \tilde{u} &= U_1 - \frac{U_1}{\varepsilon/2 + \gamma} \left( \frac{\varepsilon}{2} + f(x') + \frac{\gamma}{\sqrt{1 + |\nabla f|^2}} \right)\\
&= - \frac{U_1}{\varepsilon/2 + \gamma} \left( f(x') + \gamma \Big( \frac{1}{\sqrt{1 + |\nabla f|^2}} - 1 \Big) \right).
\end{align*}
Similarly, on $\Gamma_-$, we have
\begin{align*}
\tilde{u} + \gamma \partial_\nu \tilde{u} &= U_2 + \frac{U_2}{\varepsilon/2 + \gamma} \left( -\frac{\varepsilon}{2} + g(x') - \frac{\gamma}{\sqrt{1 + |\nabla g|^2}} \right)\\
&=  \frac{U_2}{\varepsilon/2 + \gamma} \left( g(x') - \gamma \Big( \frac{1}{\sqrt{1 + |\nabla g|^2}} - 1 \Big) \right).
\end{align*}
We denote $h = \tilde{u} + \gamma \partial_\nu \tilde{u}$ on $\Gamma_{\pm}$. By \eqref{fg_0}, \eqref{def:c_2}, and $U_1=-U_2$, we have
\begin{equation}\label{h_estimate}
|h(x)| \le \frac{C}{\gamma}|U_1||x'|^2, \quad x \in \Gamma_\pm.
\end{equation}

Let $\eta$ be a smooth cut-off function in the $x'$ variable, so that $\eta = 1$ in $\Omega_s$, $\eta =0$ outside $\Omega_{2s}$, and $|\nabla \eta| < C/s$. Multiplying $\tilde{u} \eta^2$ on both sides of $\Delta \tilde{u} = 0$ in $\Omega_{R_0}$, and integrating by parts, we have 
$$
\int_{\Omega_s} |\nabla \tilde{u}|^2  + \frac{1}{\gamma} \int_{\Gamma_s} |\tilde{u}|^2 \le C \int_{\Omega_{2s}} |\tilde{u}|^2 |\nabla \eta|^2 + \frac{1}{\gamma} \int_{\Gamma_{2s}} |\tilde u h|,
$$
which implies
\begin{equation}\label{energy1}
\int_{\Omega_s} |\nabla \tilde{u}|^2  + \frac{1}{\gamma} \int_{\Gamma_s} |\tilde{u}|^2 \le \frac{C}{s^2} \int_{\Omega_{2s}} |\tilde{u}|^2  + \frac{C}{\gamma} \int_{\Gamma_{2s}} |h|^2.
\end{equation}
By the fundamental theorem of calculus, we have
\begin{equation*}
\tilde{u}(x) = \int_{-\varepsilon/2 + g(x')}^{x_n} \partial_n \tilde{u}(x', s)\, ds + \tilde{u}(x',-\varepsilon/2 + g(x') ) \quad \mbox{for}~ x \in \Omega_{2s},
\end{equation*}
which implies
\begin{equation}\label{energy2}
\int_{\Omega_{2s}} | \tilde{u}|^2 \le C s^4 \int_{\Omega_{2s}} |\partial_n \tilde{u}|^2 + C s^2 \int_{\Gamma_{2s}} |\tilde{u}|^2.
\end{equation}
Combining \eqref{h_estimate}, \eqref{energy1}, and \eqref{energy2}, we have
$$
\int_{\Omega_s} |\nabla \tilde{u}|^2  + \frac{1}{\gamma} \int_{\Gamma_s} |\tilde{u}|^2 \le Cs^2 \int_{\Omega_{2s}} |\partial_n \tilde{u}|^2 + C \int_{\Gamma_{2s}} |\tilde{u}|^2 + \frac{C}{\gamma^3}|U_1|^2s^{n+3},
$$
where $C$ is some positive constant independent of $\gamma$. Let $N$ be a large constant to be determined later. We can choose $\gamma_0$ and $s_0$ to be sufficiently small, such that $Cs_0^2 \le 1/N$ and $C\gamma_0 \le 1/N$, then for any $\gamma \in (0, \gamma_0)$,
\begin{equation}\label{energy3}
\int_{\Omega_s} |\nabla \tilde{u}|^2  +  \frac{1}{\gamma} \int_{\Gamma_s} |\tilde{u}|^2 \le \frac{1}{N} \left(\int_{\Omega_{2s}} |\nabla \tilde{u}|^2 +  \frac{1}{\gamma} \int_{\Gamma_{2s}} |\tilde{u}|^2 \right)  + C|U_1|^2s^{n+3},
\end{equation}
where $C$ is some positive constant that may depend on $\gamma$.
For $k \in \bN$ satisfying $2^{-k}s_0 \ge \sqrt{\varepsilon}$, we denote
$$
A_k := \int_{\Omega_{2^{-k}s_0}} |\nabla \tilde{u}|^2  +  \frac{1}{\gamma} \int_{\Gamma_{2^{-k}s_0}} |\tilde{u}|^2.
$$
Then \eqref{energy3} implies that 
\begin{equation}\label{A_k}
\begin{aligned}
A_k &\le \left(\frac{1}{N}\right)^k A_0 + C |U_1|^2 \sum_{j=1}^k \left(\frac{1}{N}\right)^{k-j} 2^{-j(n+3)} s_{0}^{n+3}\\
&\le \left(\frac{1}{N}\right)^k A_0 + C |U_1|^2 2^{-k(n+1)}\sum_{j=1}^k \left(\frac{2^{n+1}}{N}\right)^{k-j} s_{0}^{n+3}.
\end{aligned}
\end{equation}
By \eqref{energy1} with $s = s_0$, we have
\begin{equation}\label{A_0}
A_0 \le C (\|u\|_{L^\infty (\Omega_{2s_0})}^2 + |U_1|^2).
\end{equation} 
Now we choose $N \ge 2^{n+1 + 2\beta}$. Then by \eqref{A_k} and \eqref{A_0},
\begin{equation}\label{A_k2}
A_k \le C2^{-k(n+1 + 2\beta)}\|u\|_{L^\infty (\Omega_{2s_0})}^2 +  C2^{-k(n+1)}|U_1|^2.
\end{equation}
By \eqref{energy2} and \eqref{A_k2}, we have
\begin{align*}
\fint_{\Omega_{2^{-k+1}s_0} \setminus \Omega_{2^{-k}s_0}}  |\tilde{u}|^2 &\le C 2^{k(n+1)}  \int_{\Omega_{2^{-k+1}s_0}} |\tilde{u}|^2 \le C 2^{k(n-1)} A_{k-1}\\
 &\le C 2^{-(2+ 2\beta) k}\|u\|_{L^\infty (\Omega_{2s_0})}^2 + C 2^{-2k}|U_1|^2,
\end{align*}
which implies 
\begin{align*}
\dashnorm{u}{L^2(\Omega_{4s}\setminus\Omega_s)} &\le \dashnorm{\tilde{u}}{L^2(\Omega_{4s}\setminus\Omega_s)} + \dashnorm{v}{L^2(\Omega_{4s}\setminus\Omega_s)}\\
&\le C s^{1+\beta} \|u\|_{L^\infty (\Omega_{2s_0})} + Cs|U_1| \quad \mbox{for}~s \in [\sqrt{\varepsilon},s_0).
\end{align*}
Then by the estimate \eqref{gradient-1/2} with $U_1 = -U_2$, we have
\begin{align*}
| \nabla u(x)| \le C|x'|^\beta \|u\|_{L^\infty (\Omega_{2s_0})} + C|U_1|\quad \mbox{for}~x \in \Omega_{s_0}\setminus\Omega_{2\sqrt{\varepsilon}}.
\end{align*}
To obtain the estimate for $x \in \Omega_{2\sqrt{\varepsilon}}$, we choose $k$ so that $2^{-k-1}s_0 < 4\sqrt{\varepsilon} \le 2^{-k}s_0$. Then \eqref{energy2} with $s =2 \sqrt{\varepsilon}$ and \eqref{A_k2} imply
\begin{align*}
\dashnorm{u}{L^2(\Omega_{4\sqrt{\varepsilon}})} &\le \dashnorm{\tilde{u}}{L^2(\Omega_{4\sqrt{\varepsilon}})} + \dashnorm{v}{L^2(\Omega_{4\sqrt{\varepsilon}})} \\
 &\le C \varepsilon^{-\frac{n-1}{4}} A_{k-1}^{1/2} + C \varepsilon\,|U_1|  \le C \varepsilon^{(\beta+1)/2} \|u\|_{L^\infty (\Omega_{2s_0})} + C \varepsilon^{1/2}|U_1|,
\end{align*}
which gives, by the estimate \eqref{gradient-1/2} with $U_1 = -U_2$ again,
\begin{align*}
| \nabla u(x)| \le C\varepsilon^{\beta/2} \|u\|_{L^\infty (\Omega_{2s_0})} + C|U_1| \quad \mbox{for}~x \in \Omega_{2\sqrt{\varepsilon}}.
\end{align*}
This concludes the proof.
\end{proof}

\section{A dimensional reduction argument}\label{sec4}

In this section, we adapt a dimensional reduction argument in \cite{DLY} to reduce the equation \eqref{narrow_region} to $n-1$ dimensions, under the assumption that $f,g$ satisfy \eqref{fg_0} and \eqref{fg_2}. This argument plays an important role in proving Theorems \ref{thm_upperbound} and \ref{main_thm}.

The weak formulation of \eqref{narrow_region} is given by
$$
\int_{\Omega_{R_0}} \nabla u \cdot \nabla \varphi \, dx + \frac{1}{\gamma} \int_{\Gamma_\pm} u \varphi \, d\sigma = 0
$$
for all smooth $\varphi$ that vanishes on $\{|x'| = R_0\}$.
Flatten the boundaries by
\begin{equation}\label{x_to_y}
\left\{
\begin{aligned}
y' &= x' ,\\
y_n &= 2 \varepsilon \left( \frac{x_n - g(x') + \varepsilon/2}{\varepsilon + f(x') - g(x')} - \frac{1}{2} \right),
\end{aligned}\right.
\quad \forall (x',x_n) \in \Omega_{R_0},
\end{equation}
and let $v(y) = u(x)$. Note that
\begin{align*}
&\int_{\Gamma_+} u(x', x_n) \varphi(x', x_n) \, d\sigma \\
=& \int_{-R_0}^{R_0} u(x', \varepsilon/2+ f(x')) \varphi(x', \varepsilon/2+ f(x')) \sqrt{1+ |\nabla_{x'}f(x')|^2} \, dx'\\
=& \int_{-R_0}^{R_0} v(y', \varepsilon) \varphi(y', \varepsilon) \sqrt{1+ |\nabla_{y'}f(y')|^2} \, dy'.
\end{align*}
Similarly,
$$
\int_{\Gamma_-} u(x', x_n) \varphi(x', x_n) \, d\sigma =  \int_{-R_0}^{R_0} v(y', -\varepsilon) \varphi(y', -\varepsilon) \sqrt{1+ |\nabla_{y'}g(y')|^2} \, dy'.
$$
Therefore, $v$ satisfies
\begin{equation}\label{equation_v}
\left\{
\begin{aligned}
-\partial_i(a^{ij}(y) \partial_j v(y)) &=0 \quad \mbox{in } Q_{R_0, \varepsilon},\\
 2\varepsilon\gamma^{-1} \sqrt{1+ |\nabla_{y'}f|^2} v(y) + a^{nj}(y) \partial_j v(y) &= 0 \quad \mbox{on } \{y_n = \varepsilon\},\\
 2\varepsilon\gamma^{-1} \sqrt{1+ |\nabla_{y'}g|^2} v(y) - a^{nj}(y) \partial_j v(y) &= 0 \quad \mbox{on } \{y_n = -\varepsilon\},
\end{aligned}
\right.
\end{equation}
where the coefficient matrix $(a^{ij}(y))$ is given by
\begin{equation}\label{a_ij_formula}
\begin{aligned}
(a^{ij}(y)) =& \frac{ 2 \varepsilon(\partial_x y)(\partial_x y)^t}{\det (\partial_x y)} \\
=& \begin{pmatrix}
\varepsilon + \mu|y'|^2 &0 &\cdots &0 &a^{1n}\\
0 &\varepsilon + \mu|y'|^2 &\cdots &0 &a^{2n}\\
\vdots &\vdots &\ddots &\vdots &\vdots\\
0 &0 &\cdots &\varepsilon + \mu|y'|^2 &a^{n-1,n}\\
a^{n1} &a^{n2} &\cdots &a^{n,n-1} &  \frac{4\varepsilon^2 + \sum_{i=1}^{n-1}|a^{in}|^2}{\varepsilon + f(y') - g(y')}
\end{pmatrix}\\
&+
\begin{pmatrix}
e^1 &0 &\cdots &0 &0\\
0 &e^2 &\cdots &0 &0\\
\vdots &\vdots &\ddots &\vdots &\vdots\\
0 &0 &\cdots & e^{n-1} &0\\
0 &0 &\cdots & 0 &0
\end{pmatrix},
\end{aligned}
\end{equation}

\begin{align*}
a^{ni} = a^{in} &= -2\varepsilon \partial_i g(y') - (y_n + \varepsilon)\partial_i(f(y') - g(y'));\\
e^{i} &= f(y')-g(y')-\mu|y'|^2 \quad \mbox{for}~i= 1,2,\ldots, n -1,
\end{align*}
and
\begin{equation}\label{Q_s_t}
Q_{s,t}:= \{ y = (y',y_n) \in \bR^n ~\big|~  |y'| < s,  |y_n| < t\}.
\end{equation}
Note that there is an extra $2\varepsilon$ factor appearing in the boundary condition of \eqref{equation_v}. This is due to the extra $2\varepsilon$ factor of $(a^{ij}(y))$ defined in \eqref{a_ij_formula}.
We define
\begin{equation}
\label{v_bar_def}
\bar{v}(y') := \fint_{-\varepsilon}^\varepsilon v(y',y_n)\, dy_n.
\end{equation}
By \eqref{equation_v}, $\bar{v}$ satisfies
\begin{equation*}
\begin{aligned}
&\dv((\varepsilon+\mu|y'|^2)\nabla \bar v)\\
=& - \frac{1}{2\varepsilon}\sum_{j=1}^{n} (a^{nj} \partial_j v)(y',\varepsilon) + \frac{1}{2\varepsilon}\sum_{j=1}^{n} (a^{nj} \partial_j v)(y',-\varepsilon) \\
&-\sum_{i=1}^{n-1}\partial_i(\overline{a^{in}\partial_n v}) - \sum_{i=1}^{n-1} \partial_i(e^i \partial_i \bar v)\\
=& \gamma^{-1}\sqrt{1+ |\nabla_{y'}f|^2} v(y', \varepsilon)+ \gamma^{-1}\sqrt{1+ |\nabla_{y'}g|^2} v(y', -\varepsilon)-\sum_{i=1}^{n-1}\partial_i(\overline{a^{in}\partial_n v}+e^i \partial_i \Bar{v}) \\
=& 2\gamma^{-1} \bar{v}(y')    \underbrace{-\sum_{i=1}^{n-1}\partial_i(\overline{a^{in}\partial_n v}+e^i \partial_i \Bar{v})}_{:= \dv F}\\
&+\underbrace{\gamma^{-1}\sqrt{1+ |\nabla_{y'}f|^2} v(y', \varepsilon) + \gamma^{-1}\sqrt{1+ |\nabla_{y'}g|^2} v(y', -\varepsilon) - 2\gamma^{-1} \bar{v}(y')}_{:= G},
\end{aligned}
\end{equation*}
where $\overline{a^{in}\partial_n v}$ denotes the vertical average of $a^{in}\partial_n v$ with respect to $y_n \in (-\varepsilon, \varepsilon)$. We will adapt this notation throughout this article.

Summarizing this section, we have shown the following lemma.

\begin{lemma}\label{lemma_reduction}
Let $f,g$ satisfy \eqref{fg_0} and \eqref{fg_2}, $u \in H^1(\Omega_{R_0})$ be a solution of \eqref{narrow_region}. Perform the change of variables \eqref{x_to_y}, and let $v(y) = u(x)$, $\bar{v}$ be defined as in \eqref{v_bar_def}. Then $\bar{v}$ satisfies $\bar{v}(0) = 0$ and the equation
\begin{equation}\label{equation_v_bar}
\dv((\varepsilon+\mu|y'|^2)\nabla \bar v) - 2\gamma^{-1} \bar{v} = \dv F + G \quad \mbox{in}~B_{R_0} \subset \bR^{n-1},
\end{equation}
where
\begin{equation}\label{FG_bound}
\begin{aligned}
|F(y')| &\le C\Big(\varepsilon|y'| |\overline{\partial_n v}| + |e^1| |\nabla_{y'}\overline{v}|\Big), \\
 |G(y')| &\le C\Big(\varepsilon \max_{y_n \in (-\varepsilon, \varepsilon)} |\partial_n v| + |y'|^2 |\overline{v}| \Big),
\end{aligned}
\end{equation}
and $C$ is some positive constant depending only on $n$, $\gamma$, $\|f\|_{C^{2,\sigma}}$, and $\|g\|_{C^{2,\sigma}}$.
\end{lemma}

\begin{proof}
The derivation of equation \eqref{equation_v_bar} follows from the argument above. It remains to prove the estimates \eqref{FG_bound}. Since $f,g$ satisfy \eqref{fg_0}, we have $|a^{in}| \le C \varepsilon |y'|$ for $1 \le i \le n-1$. Therefore, it is straightforward to see that
\begin{align*}
|F(y')| \le C\Big(|\overline{a^{in}\partial_n v}| + |e^i \partial_i \bar v)|\Big) \le C\Big(\varepsilon|y'| |\overline{\partial_n v}| + |e^1| |\nabla_{y'}\overline{v}|\Big).
\end{align*}
For the $G$ term, we have
\begin{align*}
|G(y')| \le & \gamma^{-1}\sqrt{1+ |\nabla_{y'}f|^2} |v(y', \varepsilon) - \bar{v}(y')| + \gamma^{-1}\sqrt{1+ |\nabla_{y'}g|^2} |v(y', -\varepsilon) - \bar{v}(y')|\\
&\gamma^{-1} \Big(\sqrt{1+ |\nabla_{y'}f|^2} + \sqrt{1+ |\nabla_{y'}g|^2} -2 \Big)|\bar{v}(y')|\\
\le& C\Big(\varepsilon \max_{y_n \in (-\varepsilon, \varepsilon)} |\partial_n v| + |y'|^2 |\overline{v}| \Big).
\end{align*}
This completes the proof.
\end{proof}

\section{Proof of Theorem \ref{thm_upperbound}}\label{sec5}
In this section, we prove Theorem \ref{thm_upperbound}. Without loss of generality, we may assume $\mu = 1$. Namely, we consider
\begin{equation*}
f(x')- g(x') =|x'|^2 + O(|x'|^{2+\sigma}) \quad\mbox{for}~~0<|x'|<R_0.
\end{equation*}
In this case,
\begin{equation}
\label{alpha_new}
\alpha = \frac {-(n-1)+\sqrt{(n-1)^2+4(n-2 +  2/\gamma)}}{2}.
\end{equation}
For general $\mu > 0$, we only need to replace $\varepsilon$ and $\gamma$ in the proof by $\varepsilon/\mu$ and $\mu\gamma$ respectively, in view of equation \eqref{equation_v_bar}. The strategy of the proof is similar to the proof of \cite{DLY}*{Theorem 1.1}. However, some modifications are needed to address the case $n = 2$ and the extra zero-order term $-2\gamma^{-1}\bar{v}$ in the equation \eqref{equation_v_bar}.

For $\varepsilon > 0$, $\sigma, s \in \bR$, an open set $D \in \bR^{n-1}$, we introduce the following norm
\begin{equation}\label{def_norm}
\| F \|_{\varepsilon, \sigma,s,D}: = \sup_{y' \in D} \frac{|F(y')|}{|y'|^{\sigma} (\varepsilon + |y'|^2)^{1-s}}.
\end{equation}
Throughout this section, $B_R$ always lies in $\bR^{n-1}$. For any $0 < R < R_0$, we denote
$$
(u)_{\partial B_R}:= \fint_{\partial B_R} u \, dS
$$
to be the average of $u$ over $\partial B_R$. 

First, we establish estimates for $\bar{v}$.

\begin{proposition}
\label{prop_grad_v_bar_control}
For $n \ge 2$, $0 \le s \le 1/2$, $\gamma > 0$, $\sigma> 0$, $1+\sigma - 2s \neq \min\{\alpha,1\}$, and $\varepsilon > 0$,
let $\bar{v} \in H^1(B_{1})$ be a solution of
\begin{equation*}
\dv((\varepsilon+|y'|^2)\nabla \bar v) -2\gamma^{-1}\bar{v} = \dv F + G \quad\text{in}\,\,B_{R_0}\subset \bR^{n-1},
\end{equation*}
where $F,G$ satisfy
\begin{equation}
\label{FG_bound3}
\|F\|_{\varepsilon, \sigma,s,B_{R_0}}< \infty , \quad \|G\|_{\varepsilon, \sigma-1,s,B_{R_0}} < \infty.
\end{equation}
Then for any $R \in (0, R_0/2)$, we have
\begin{equation}
\label{v_bar_L2_sphere}
\begin{aligned}
&\dashnorm{\bar{v} - (\bar{v}_{\partial B_R})}{L^2(\partial B_R)}\\
&\le C(\|F\|_{\varepsilon, \sigma,s,B_{R_0}} + \|G\|_{\varepsilon, \sigma-1,s,B_{R_0}} +  \dashnorm{\bar v - (\bar{v})_{\partial B_{R_0}}}{L^2(\partial B_{R_0})})  R^{\tilde\alpha},
\end{aligned}
\end{equation}
where $\tilde\alpha := \min \{ \alpha, 1, 1 + \sigma - 2s \}$, $\alpha$ is given in \eqref{alpha_new}, and $C$ is some positive constant depending only on $n$, $\sigma$, $s$, and is independent of $\varepsilon$.
\end{proposition}

For the proof, we use an iteration argument based on the following lemmas.

\begin{lemma}
\label{lemma_v1}
For $n \ge 2$, $\gamma > 0$, and $\varepsilon > 0$, let $v_1 \in H^1(B_{R_0})$ satisfy
\begin{equation}
\label{equation_v1}
\dv((\varepsilon+|y'|^2)\nabla v_1) -2\gamma^{-1}v_1=0\quad\text{in}\,\,B_{R_0} \subset \bR^{n-1}.
\end{equation}
Then for any $0 < \rho < R \le R_0$, we have
$$
\left( \fint_{\partial B_\rho} |v_1(y') - (v_1)_{\partial B_\rho}|^2 \, dS\right)^{\frac{1}{2}} \le \left(\frac{\rho}{R} \right)^{\widehat\alpha} \left( \fint_{\partial B_R} |v_1(y') - (v_1)_{\partial B_R}|^2 \, dS \right)^{\frac{1}{2}},
$$
where $\widehat{\alpha}:= \min\{\alpha,1 \}$, and $\alpha$ is given in \eqref{alpha_new}.
\end{lemma}

\begin{proof}
By the elliptic theory, $v_1 \in C^\infty(B_{1})$. By scaling, it suffices to prove the lemma for $R = 1$. We first provide a proof for the case when $n \ge 3$. Denote $y' = (r ,\xi) \in (0,1) \times \bS^{n-2}$. We can rewrite \eqref{equation_v1} as
$$
\partial_{rr} v_1 +\left(\frac {n-2}r +\frac{2r}{\varepsilon+r^2} \right)\partial_{r} v_1 +\frac 1 {r^2}\Delta_{\bS^{n-2}} v_1 - \frac{2}{\gamma(\varepsilon+r^2)}v_1 =0 \quad \mbox{in}~~B_1 \setminus \{0\}.
$$
Take the spherical harmonic decomposition
\begin{equation}
\label{v1_expansion}
v_1(y') = V_0(r) Y_0 + \sum_{k=1}^\infty \sum_{i=1}^{N(k)} V_{k,i}(r)Y_{k,i}(\xi), \quad y' \in B_1\setminus\{0\},
\end{equation}
where $Y_0 = |\bS^{n-2}|^{-1/2}$, $Y_{k,i}$ is a $k$-th degree spherical harmonics, that is,
$$
-\Delta_{\bS^{n-2}} Y_{k,i} = k(k+n-3)Y_{k,i}
$$ and $\{Y_{k,i}\}_{k,i} \cup Y_0$ forms an orthonormal basis of $L^2(\bS^{n-2})$.
Then
$$
V_0(r) Y_0 = \fint_{\bS^{n-2}} v_1(r,\xi) \, d\xi = (v_1)_{\partial B_r},
$$
$V_{k,i}(r) \in C^2(0,1)$ is given by
$$
V_{k,i}(r) = \int_{\bS^{n-2}} v_1(y') Y_{k,i}(\xi) \, d\xi
$$
and satisfies
\begin{equation*}
\begin{aligned}
L_k V_{k,i}:=& V_{k,i}''(r) +\left(\frac {n-2}r +\frac{2r}{\varepsilon+r^2} \right)V_{k,i}'(r)\\
 &- \left( \frac {k(k+n-3)} {r^2} + \frac{2}{\gamma(\varepsilon + r^2)} \right) V_{k,i}(r)=0 \quad \mbox{in}~~(0,1)
\end{aligned}
\end{equation*}
for each $k \in \bN$, $i = 1, 2 ,\ldots, N(k)$. Since $v_1 \in C^\infty(B_1)$ and  $\{Y_{k,i}\}_{k,i} \cup Y_0$ is an orthonormal basis, we have
\begin{align*}
\sum_{k=1}^\infty \sum_{i=1}^{N(k)}  |V_{k,i}(r)|^2 = \fint_{\partial B_\rho} |v_1(y') - (v_1)_{\partial B_r}|^2 \, dS \to 0 \quad \mbox{as}~r \to 0.
\end{align*}
This implies 
$$
|V_{k,i}(r)| \to 0 \quad \mbox{as}~r \to 0
$$
for each $k \in \bN$, $i = 1, 2 ,\ldots, N(k)$.

Next, we consider two separate cases $0< \gamma \le 1$ and $\gamma > 1$.

\textit{Case 1:} When $0< \gamma \le 1$, for any $k \in \bN$, by a direct computation,
$$
L_k r = \frac{n-2}{r} - \frac{k(k+n-3)}{r} + \left(1- \frac{1}{\gamma} \right) \frac{2r}{\varepsilon + r^2} \le 0.
$$
Since
$$
\pm V_{k,i}(r) - |V_{k,i}(1)|r \to 0 \quad \mbox{as}~~r \searrow 0,
$$
and clearly,
$$
\pm V_{k,i}(r) - |V_{k,i}(1)|r \le 0 \quad \mbox{when}~~r =1.
$$
By the maximum principle,
\begin{equation}
\label{V_decay_1}
|V_{k,i}(r)|\le r |V_{k,i}(1)| \quad \mbox{for}~~0<r\le1.
\end{equation}

\textit{Case 2:} When $\gamma > 1$, for any $k \in \bN$, let
$$
\alpha_k:= \frac {-(n-1)+\sqrt{(n-1)^2+4[k(k+n-3)+2/\gamma]}}{2} > 0.
$$
Note that $\alpha_1=\alpha$, where $\alpha$ is defined in \eqref{alpha_new}.
By a direct computation,
\begin{align*}
\begin{aligned}
L_kr^{\alpha_k} =& \alpha_k(\alpha_k-1)r^{\alpha_k-2} + \left( n-2 + \frac{2r^2}{\varepsilon + r^2} \right) \alpha_k r^{\alpha_k-2}\\
&- \left( k(k+n-3) + \frac{2}{\gamma} \frac{r^2}{\varepsilon+r^2} \right)r^{\alpha_k-2}\\
=& - 2 \varepsilon \left( \alpha_k - \frac{1}{\gamma} \right) \frac{ r^{\alpha_k - 2}}{\varepsilon+r^2}  \quad\quad\mbox{for}~r\in(0,1),
\end{aligned}
\end{align*}
where in the second line, we used the equality
\begin{equation}\label{alpha_k_gamma_equality}
\alpha_k^2 + (n-1) \alpha_k - (k(k+n-3) + 2\gamma^{-1}) = 0
\end{equation}
followed from the definition of $\alpha_k$. By \eqref{alpha_k_gamma_equality} again, we have
\begin{align*}
\alpha_k - \frac{1}{\gamma} &= \alpha_k - \frac{1}{2} \Big(\alpha_k^2 + (n-1) \alpha_k - k(k+n-3) \Big)\\
&= - \frac{1}{2} (\alpha_k + k+n-3) (\alpha_k - k) > 0,
\end{align*}
since
\begin{align*}
\alpha_k &< \frac {-(n-1)+\sqrt{(n-1)^2+4k^2 + 4(n-1)k + 8 - 8k}}{2}\\
&\le \frac {-(n-1)+\sqrt{(n-1 + 2k)^2}}{2} \le k.
\end{align*}
Therefore, 
$$
L_kr^{\alpha_k} < 0\quad\quad\mbox{for}~r\in(0,1).
$$
Since 
$$
\pm V_{k,i}(r) - |V_{k,i}(1)|r^{\alpha_k} \to 0 \quad \mbox{as}~~r \searrow 0,
$$
and clearly,
$$
\pm V_{k,i}(r) - |V_{k,i}(1)|r^{\alpha_k} = 0 \quad \mbox{when}~~r =1.
$$
By the maximum principle,
\begin{equation}
\label{V_decay_2}
|V_{k,i}(r)|\le r^{\alpha_k} |V_{k,i}(1)| \quad \mbox{for}~~0<r<1.
\end{equation}

Combining these two cases, it follows from \eqref{v1_expansion},\eqref{V_decay_1}, and \eqref{V_decay_2},
\begin{align*}
\fint_{\partial B_\rho} |v_1(y') - (v_1)_{\partial B_\rho}|^2 \, dS &= \sum_{k=1}^\infty \sum_{i=1}^{N(k)}  |V_{k,i}(\rho)|^2 \\
&\le \rho^{2\widehat\alpha} \sum_{k=1}^\infty \sum_{i=1}^{N(k)}|V_{k,i}(1)|^2 \\
&= \rho^{2\widehat\alpha} \fint_{\partial B_1} |v_1(y')- (v_1)_{\partial B_1}|^2 \, d\sigma.
\end{align*}
This completes the proof for $n \ge 3$. 

When $n = 2$, we use $(\bar{v}(r) - \bar{v}(-r))/2$ in place of $V_{1,1}(r)$, and repeat the arguments for the estimates \eqref{V_decay_1} and \eqref{V_decay_2} to obtain
$$
\left| \frac{\bar{v}(\rho) - \bar{v}(-\rho)}{2} \right| \le \rho^{\widehat{\alpha}} \left| \frac{\bar{v}(1) - \bar{v}(-1)}{2} \right|.
$$
The lemma is proved.
\end{proof}

\begin{lemma}
\label{lemma_v2}
For $n \ge 2$, $s \le 1/2$, $\sigma > 0$, and $\varepsilon > 0$, suppose that $F,G$ satisfy \eqref{FG_bound3}, and $v_2 \in H^1_0(B_{1})$ satisfies
\begin{equation}
\label{equation_v2}
\dv((\varepsilon+|y'|^2)\nabla v_2) - 2\gamma^{-1} v_2=\dv F + G\quad\text{in}\,\,B_{1}\subset \bR^{n-1}.
\end{equation}
Then we have
\begin{equation}\label{v2_L^infty}
\|v_2\|_{L^\infty(B_{1})}\le C(\|F\|_{\varepsilon,\sigma,s, B_{1}} + \|G\|_{\varepsilon,\sigma-1,s,B_{1}}),
\end{equation}
where $C>0$ depends only on $n$, $\sigma$, and $s$, and is in particular independent of $\varepsilon$.
\end{lemma}

\begin{proof}
Without loss of generality, we assume
$$\|F\|_{\varepsilon,\sigma,s, B_{1}} + \|G\|_{\varepsilon,\sigma-1,s,B_{1}} = 1.$$ 
When $\varepsilon \ge 1$, we divide the equation \eqref{equation_v2} by $\varepsilon$ and obtain
$$
\dv((1+\varepsilon^{-1}|y'|^2)\nabla v_2) - 2(\varepsilon\gamma)^{-1} v_2=\varepsilon^{-1}\dv F + \varepsilon^{-1}G\quad\text{in}\,\,B_{1}\subset \bR^{n-1}.
$$
By assumptions,
$$
|\varepsilon^{-1}F| \le C|y'|^\sigma, \quad |\varepsilon^{-1}G| \le C|y'|^{\sigma-1},
$$
so that \eqref{v2_L^infty} follows from the standard Moser iteration. Therefore, we will focus on the case $0 < \varepsilon < 1$, where we will use Moser iteration in weighted $L^p$ spaces.

For $p \ge 2$, we multiply the equation \eqref{equation_v2} with $-|v_2|^{p-2}v_2$ and integrate by parts to obtain
\begin{equation}\label{energy}
\begin{aligned}
&(p-1) \int_{B_{1}}(\varepsilon + |y'|^2) |\nabla v_2|^{2}|v_2|^{p-2}\,dy' + 2\gamma^{-1}\int_{B_{1}} |v_2|^p \, dy' \\
&= (p-1)\int_{B_{1}}F \cdot \nabla v_2|v_2|^{p-2}\,dy' - \int_{B_{1}} G v_2 |v_2|^{p-2} \, dy'.
\end{aligned}
\end{equation}
By the definition in \eqref{def_norm},
$$
\begin{aligned}
|F(y')| &\le |y'|^{\sigma}(\varepsilon + |y'|^2)^{1-s}\|F\|_{\varepsilon,\sigma,s, B_{1}},\\|G(y')| &\le |y'|^{\sigma-1}(\varepsilon + |y'|^2)^{1-s} \|G\|_{\varepsilon,\sigma-1,s,B_{1}} \quad \mbox{for}~~y' \in B_1.
\end{aligned}
$$
To estimate the first term on the right-hand side of \eqref{energy}, we use Young's inequality, $s \le 1/2$,  and $\varepsilon < 1$,
\begin{align*}
&(p-1)\left|\int_{B_{1}}F \cdot \nabla v_2|v_2|^{p-2}\,dy' \right|\\
\le& \frac{(p-1)}{4}\int_{B_{1}}(\varepsilon + |y'|^2) |\nabla v_2|^{2}|v_2|^{p-2}\,dy'\\
&+  C(p-1)\int_{B_{1}}|y'|^{2\sigma}(\varepsilon + |y'|^2)^{1-2s} |v_2|^{p-2}\,dy'\\
\le& \frac{(p-1)}{4}\int_{B_{1}}(\varepsilon + |y'|^2) |\nabla v_2|^{2}|v_2|^{p-2}\,dy' +  C(p-1)\int_{B_{1}}|y'|^{2\sigma}|v_2|^{p-2}\,dy'.
\end{align*}
Then by H\"older's inequality and $\sigma > 0$,
\begin{equation}\label{Holder1}
\begin{aligned}
\int_{B_{1}}|y'|^{2\sigma}|v_2|^{p-2}\,dy' \le& \left( \int_{B_{1}}|y'|^{\sigma(n+1+2\varrho) - (n-1 + 2\varrho)} \, dy' \right)^{\frac{2}{n+1+2\varrho}} \times \\
& \times \left( \int_{B_{1}} |y'|^2 |v_2|^{(p-2) \frac{n+1+2\varrho}{n-1+2\varrho}}  \, dy'\right)^{\frac{n-1+2\varrho}{n+1+2\varrho}},
\end{aligned}
\end{equation}
where $\varrho > 0$ is chosen sufficiently small so that
$$\int_{B_{1}}|y'|^{\sigma(n+1+2\varrho)/2 - (n-1 + 2\varrho)} \, dy' < \infty.$$
To estimate the last term of \eqref{energy}, we have
\begin{align*}
\left| \int_{B_{1}} G v_2 |v_2|^{p-2} \, dy' \right| \le& \int_{B_{1}}|y'|^{\sigma-1}(\varepsilon + |y'|^2)^{1-s} |v_2|^{p-1}\, dy'\\
=& \frac{1}{n-1}  \int_{B_{1}}|y'|^{\sigma-1}(\varepsilon + |y'|^2)^{1-s} |v_2|^{p-1}(\nabla \cdot y')\, dy'\\
=& - \frac{\sigma-1}{n-1} \int_{B_{1}}|y'|^{\sigma-1}(\varepsilon + |y'|^2)^{1-s} |v_2|^{p-1}\, dy'\\
&- \frac{2(1-s)}{n-1} \int_{B_{1}}|y'|^{\sigma+1}(\varepsilon + |y'|^2)^{-s} |v_2|^{p-1}\, dy'\\
&-\frac{p-1}{n-1} \int_{B_{1}}|y'|^{\sigma-1}(\varepsilon + |y'|^2)^{1-s} (y' \cdot \nabla v_2) |v_2|^{p-3} v_2 \, dy.
\end{align*}
Therefore,
\begin{align*}
&\int_{B_{1}}|y'|^{\sigma-1}(\varepsilon + |y'|^2)^{1-s} |v_2|^{p-1}\, dy'\\
\le & C(p-1)\int_{B_{1}} |y'|^{\sigma}(\varepsilon + |y'|^2)^{1-s} |\nabla v_2| |v_2|^{p-2}\, dy'\\
&+ C \int_{B_{1}}|y'|^{\sigma+1}(\varepsilon + |y'|^2)^{-s} |v_2|^{p-1}\, dy'\\
\le &\frac{(p-1)}{4}\int_{B_{1}}(\varepsilon + |y'|^2) |\nabla v_2|^{2}|v_2|^{p-2}\,dy'\\
&+  C(p-1)\int_{B_{1}}|y'|^{2\sigma}|v_2|^{p-2}\,dy' + C\int_{B_{1}}|y'|^{\sigma}|v_2|^{p-1}\,dy',
\end{align*}
where we used $\varepsilon < 1$ in the second inequality. The first term in the last line can be controlled as in \eqref{Holder1}, and the second term can be estimated similarly:
\begin{align*}
\left| \int_{B_{1}}|y'|^{\sigma}|v_2|^{p-1}\,dy' \right| \le& \left( \int_{B_{1}}|y'|^{\sigma(n+1+2\varrho)/2 - (n-1 + 2\varrho)} \, dy' \right)^{\frac{2}{n+1+2\varrho}} \times \\
& \times \left( \int_{B_{1}} |y'|^2 |v_2|^{(p-1) \frac{n+1+2\varrho}{n-1+2\varrho}}  \, dy'\right)^{\frac{n-1+2\varrho}{n+1+2\varrho}}.
\end{align*}
Hence, from \eqref{energy}, we have
\begin{equation}
\begin{aligned}
\label{moser_inequality1}
&\frac{4(p-1)}{p^2} \int_{B_{1}} |y'|^2 \left| \nabla |v_2|^{\frac{p}{2}} \right|^2\,dy' = (p-1)\int_{B_{1}} |y'|^2 |\nabla v_2|^{2}|v_2|^{p-2}\,dy'\\
&\le C(p-1)  \|v_2^{p-2}\|_{L^{\frac{n+1+2\varrho}{n-1+2\varrho}}(B_{1}, |y'|^2dy')} + C \|v_2^{p-1}\|_{L^{\frac{n+1+2\varrho}{n-1+2\varrho}}(B_{1}, |y'|^2dy')}.
\end{aligned}
\end{equation}
We use the following version of the Caffarelli-Kohn-Nirenberg inequality (see \cite{CKN}):
\begin{equation}
\label{CKN_inequality}
 \|u\|_{L^{\frac{2(n+1)}{n-1}}(B_{1}, |y'|^2dy')} \le C  \|\nabla u\|_{L^{2}(B_{1}, |y'|^2dy')} \quad \forall u \in H_0^1(B_1, |y'|^2dy').
\end{equation}
Taking $p = 2$ in \eqref{moser_inequality1}, we have, by \eqref{CKN_inequality} with $u = |v_2|$ and H\"older's inequality,
\begin{equation}
\label{starting_point}
\|v_2\|_{L^{\frac{2(n+1+2\varrho)}{n-1+2\varrho}}(B_{1}, |y'|^2dy')} \le C.
\end{equation}
For $p \ge 2$, from \eqref{moser_inequality1}, by \eqref{CKN_inequality} with $u = |v_2|^{\frac{p}{2}}$ and H\"older's inequality,
\begin{align*}
&\|v_2\|_{L^{\frac{(n+1)p}{n-1}}(B_{1}, |y'|^2dy')}^p \le C\| \nabla |v_2|^{\frac{p}{2}} \|_{L^{2}(B_{1}, |y'|^2dy')}^{2}\\
&\le C p^2 \|v_2\|_{L^{\frac{n+1+2\varrho}{n-1+2\varrho}(p-2)}(B_{1}, |y'|^2dy')}^{p-2} + Cp\|v_2\|_{L^{\frac{n+1+2\varrho}{n-1+2\varrho}(p-1)}(B_{1}, |y'|^2dy')}^{p-1}\\
&\le \max_{i=1,2} Cp^i \|v_2\|_{L^{\frac{n+1+2\varrho}{n-1+2\varrho}p}(B_{1}, |y'|^2dy')}^{p-i}.
\end{align*}
By Young's inequality,
\begin{align*}
\|v_2\|_{L^{\frac{(n+1)p}{n-1}}(B_{1}, |y'|^2dy')} &\le \max_{i=1,2} (Cp^i)^{1/p} \left( \frac{p-i}{p} \|v_2\|_{L^{\frac{n+1+2\varrho}{n-1+2\varrho}p}(B_{1}, |y'|^2dy')} + \frac{i}{p} \right)\\
&\le (Cp^2)^{1/p} \left( \|v_2\|_{L^{\frac{n+1+2\varrho}{n-1+2\varrho}p}(B_{1}, |y'|^2dy')} + \frac{2}{p} \right).
\end{align*}
For $k \ge 0$, let
$$
p_k = 2 \left( \frac{n+1}{n-1} \cdot \frac{n-1+2\varrho}{n+1+2\varrho} \right)^k \frac{n+1+2\varrho}{n-1+2\varrho} \nearrow + \infty \quad \mbox{as}~k \to + \infty.
$$
Iterating the relations above, we have, by \eqref{starting_point},
\begin{equation}
\begin{aligned}
\label{Moser_iteration}
\|v_2\|_{L^{p_k}(B_{1}, |y'|^2dy')} &\le \prod_{i=0}^{k-1} \left( Cp_i^2 \right)^{3/p_i}\|v_2\|_{L^{p_0}(B_{1}, |y'|^2dy')}\\
&+ \sum_{i=0}^{k-1} \prod_{j=i}^{k-1} \left( Cp_{j}^2 \right)^{3/p_{j}}\frac{6}{p_i} \\
&\le C \|v_2\|_{L^{\frac{2(n-1+2\varrho)}{n-3+2\varrho}}(B_{1}, |y'|^2dy')} + C \sum_{i=0}^{k-1} \frac{1}{p_i} \le C,
\end{aligned}
\end{equation}
where $C$ is a positive constant depending on $n$ and $\sigma$, and is in particular independent of $k$. The lemma is concluded by taking $k \to \infty$ in \eqref{Moser_iteration}.
\end{proof}

Now we are in a position to prove Proposition \ref{prop_grad_v_bar_control}.

\begin{proof}[Proof of Proposition \ref{prop_grad_v_bar_control}] Without loss of generality, we assume that
$$
\|F\|_{\varepsilon, \sigma,s,B_{R_0}} + \|G\|_{\varepsilon, \sigma-1,s,B_{R_0}} +  \dashnorm{\bar v - (\bar{v})_{\partial B_{R_0}}}{L^2(\partial B_{R_0})} = 1.
$$
We denote
$$
\omega(\rho):= \dashnorm{\bar{v} - (\bar{v})_{\partial B_{\rho}}}{L^2(\partial B_\rho)}.
$$
For $0 < \rho \le R/2 \le R_0/2$, we write $\bar v=v_1+v_2$ in $B_R$, where $v_2$ satisfies
$$
\dv((\varepsilon+|y'|^2)\nabla v_2) - 2\gamma^{-1}v_2=\dv F + G\quad\text{in}\,\,B_R
$$
and $v_2=0$ on $\partial B_R$. Thus $v_1$ satisfies
$$
\dv((\varepsilon+|y'|^2)\nabla v_1)- 2\gamma^{-1}v_1=0\quad\text{in}\,\,B_R,
$$
and $v_1=\bar v$ on $\partial B_R$.
Since $\tilde v_2(y'):=v_2(Ry')$ satisfies
$$
\dv((R^{-2}\varepsilon+|y'|^2)\nabla \tilde v_2) - 2\gamma^{-1}\tilde v_2
=\dv \tilde F + \tilde G\quad\text{in}\,\,B_1,
$$
where $\tilde F(y'):=R^{-1}F(Ry')$ and $\tilde{G}(y') := G(Ry')$ satisfy
\begin{align*}
\|\tilde{F}\|_{R^{-2}\varepsilon,\sigma,s,B_1} &= R^{1+ \sigma - 2s} \|F\|_{\varepsilon, \sigma,s,B_{R}}, \\
\|\tilde G\|_{R^{-2}\varepsilon, \sigma-1,s,B_{1}} &= R^{1+ \sigma - 2s} \|G\|_{\varepsilon, \sigma-1,s,B_{R}},
\end{align*}
we apply Lemma \ref{lemma_v2} to $\tilde{v}_2$ with $\varepsilon$ replaced with $R^{-2}\varepsilon$ to obtain
\begin{equation}
\label{v2_control}
\|v_2\|_{L^\infty(B_R)}\le CR^{1+\sigma-2s}.
\end{equation}
By Lemma \ref{lemma_v1},
\begin{equation}
\label{v1_control}
\left( \fint_{\partial B_\rho} |v_1(y') - (v_1)_{\partial B_\rho}|^2 \, dS \right)^{\frac{1}{2}} \le \left(\frac{\rho}{R} \right)^{\widehat\alpha} \left( \fint_{\partial B_R} |v_1(y') - (v_1)_{\partial B_R}|^2 \, dS \right)^{\frac{1}{2}}.
\end{equation}
Combining \eqref{v1_control} and \eqref{v2_control} yields, using $\bar{v} = v_1$ on $\partial B_R$ and $\bar{v} = v_1 + v_2$,
\begin{equation}
\begin{aligned}
\label{omega_iteration}
\omega(\rho)&\le \left( \fint_{\partial B_\rho} |v_1(y') - (v_1)_{\partial B_\rho}|^2 \, dS \right)^{\frac{1}{2}} + \left( \fint_{\partial B_\rho} |v_2(y') - (v_2)_{\partial B_\rho}|^2 \, dS \right)^{\frac{1}{2}}\\
&\le \left(\frac{\rho}{R} \right)^{\widehat\alpha} \left( \fint_{\partial B_R} |v_1(y') -  (v_1)_{\partial B_R}|^2 \, dS \right)^{\frac{1}{2}} + 2\|v_2\|_{L^\infty(B_R)}\\
&\le \left(\frac{\rho}{R} \right)^{\widehat\alpha} \omega(R) + CR^{1+ \sigma - 2s}.
\end{aligned}
\end{equation}
For a positive integer $k$, we take $\rho = 2^{-i-1}$ and $R = 2^{-i}$ in \eqref{omega_iteration}, and iterate from $i = 0$ to $k-1$. We have, using $1 + \sigma - 2s \neq \widehat\alpha$,
\begin{align*}
\omega(2^{-k}) &\le 2^{-k\widehat\alpha} \omega(1) + C\sum_{i=1}^k 2^{-(k-i)\widehat\alpha} (2^{1-i})^{1+\sigma - 2s}\\
&\le 2^{-k\widehat\alpha} \omega(1) + C2^{-k\widehat\alpha}  \frac{1 - 2^{k(\widehat\alpha -1-\sigma+2s)}}{1 - 2^{\widehat\alpha - 1 - \sigma + 2s}}.
\end{align*}
It follows that
$$
\omega(2^{-k}) \le 2^{-k\tilde\alpha} \left( \omega(1) + C\right),
$$
where $\tilde{\alpha} = \min\{\widehat{\alpha}, 1 + \sigma - 2s \}$.
For any $\rho \in (0, 1/2)$, let $k$ be the integer such that $2^{-k-1} < \rho \le 2^{-k}$. Then
\begin{equation*}
\omega(\rho) \le C \rho^{\tilde\alpha}, \quad \forall \rho \in (0, 1/2).
\end{equation*}
Therefore, \eqref{v_bar_L2_sphere} is proved.
\end{proof}

\begin{proof}[Proof of Theorem \ref{thm_upperbound}]
Without loss of generality, we assume that $\|u\|_{L^\infty(\Omega_{R_0})} = 1$. Recall that $\mu = 1$. We will prove \eqref{grad_u_upperbound1} and \eqref{grad_u_upperbound2} with $\alpha$ given as in \eqref{alpha_new}. We make the change of variables \eqref{x_to_y}, and let $v(y) = u(x)$. Then $v$ satisfies \eqref{equation_v}. Let $\bar{v}$ be the vertical average defined as in \eqref{v_bar_def}. By \eqref{gradient-1/2} with $U_1 = U_2 = 0$, 
\begin{equation}\label{grad_v_bar_start}
\|\nabla \bar{v}\|_{\varepsilon,0,s_0+1,B_{R_0}} \le C,
\end{equation}
where $s_0 = \frac{1}{2}$. On the other hand, since $\partial_\nu u =  - \gamma^{-1}u$ on $\Gamma_+$ and $\Gamma_-$, we have, by
\eqref{fg_0}, \eqref{u_C1_outside}, and \eqref{gradient-1/2} with $U_1 = U_2 = 0$ that
$$|\partial_n u(x)| \le C\left| \sum_{i = 1}^{n-1}x_i \partial_i u \right| + C|u| \le C, \quad \forall x \in \Gamma_+ \cup \Gamma_-.$$
By the harmonicity of $\partial_n u$, the estimate \eqref{u_C1_outside}, and the maximum principle,
\begin{equation*}
|\partial_n u| \le C \quad \mbox{in}~~\Omega_{R_0},
\end{equation*}
and consequently,
\begin{equation}
\label{grad_n_v_bound}
|\partial_n v| \le C(\varepsilon+|y'|^2)/\varepsilon \quad \mbox{in}~~Q_{R_0,\varepsilon},
\end{equation}
where $Q_{R_0,\varepsilon}$ is defined as in \eqref{Q_s_t}. On the other hand, $e^i$ in \eqref{a_ij_formula} can be bounded by $C|y'|^{2+\sigma}$. Therefore, by Lemma \ref{lemma_reduction}, \eqref{grad_v_bar_start}, and \eqref{grad_n_v_bound}, $\bar{v}$ satisfies the equation \eqref{equation_v_bar} with $\mu = 1$, $F,G$ satisfying
$$\|F\|_{\varepsilon, \sigma, s_0, B_{R_0}} + \|G\|_{\varepsilon, \sigma-1, s_0, B_{R_0}} \le C.$$
By \eqref{grad_n_v_bound},
\begin{equation}
\label{v-v_bar}
|v(y', y_n) - \bar{v}(y')| \le 2\varepsilon \max_{y_n\in (- \varepsilon , \varepsilon)} |\partial_n v(y',y_n)| \le C (\varepsilon + |y'|^2) \quad \mbox{in}~~Q_{R_0,\varepsilon}.
\end{equation}
By decreasing $\sigma$ if necessary, we may assume that $1+\sigma-2s_0=\sigma<\widehat\alpha$, where $\widehat{\alpha} = \min\{\alpha , 1\}$. Since $u$ is odd in $x_j$ for some $1 \le j \le n-1$, $\bar{v}$ is also odd in $x_j$. In particular, this implies
$$
(\bar{v})_{\partial B_R} = 0 \quad \forall R \in(0,R_0).
$$
By Proposition \ref{prop_grad_v_bar_control} and \eqref{v-v_bar}, we have
\begin{equation}
\begin{aligned}\label{v_L2_control}
\dashnorm{u}{L^2(\Omega_{2\varepsilon^{1/2}})} \le& C \dashnorm{v}{L^2(Q_{2\varepsilon^{1/2} , \varepsilon})} \\
 \le& C( \dashnorm{v-\bar{v}}{L^2(Q_{2\varepsilon^{1/2} , \varepsilon})} + \dashnorm{\bar v}{L^2(Q_{2\varepsilon^{1/2} , \varepsilon})})\\
 \le& C \varepsilon^{\tilde\alpha/2},
\end{aligned}
\end{equation}
where $\tilde{\alpha} = \min\{\widehat\alpha , 1+ \sigma - 2s_0 \} = \sigma$. By \eqref{v_L2_control} and \eqref{gradient-1/2} with $U_1 = U_2 = 0$, we have
$$
\| \nabla u\|_{L^\infty (\Omega_{\varepsilon^{1/2}})} \le C\varepsilon^{\frac{\tilde\alpha-1}{2}}.
$$
Similarly, for any $R \in (\varepsilon^{1/2}, R_0/4)$, Proposition \ref{prop_grad_v_bar_control} and \eqref{v-v_bar} implies
$$
\dashnorm{u}{L^2(\Omega_{4R}\setminus \Omega_{R/2})} \le C R^{\tilde{\alpha}}, 
$$
which implies, by \eqref{gradient-1/2} with $U_1 = U_2 = 0$,
$$
\| \nabla u\|_{L^\infty (\Omega_{4R}\setminus \Omega_{R/2})} \le CR^{\tilde\alpha-1}
$$
for any $R \in (\varepsilon^{1/2}, R_0/4)$. Therefore, we have improved the upper bound $|\nabla u(x)| \le C(\varepsilon + |x'|^2)^{-s_0}$ to $|\nabla u(x)| \le C(\varepsilon + |x'|^2)^{\frac{\tilde{\alpha}-1}{2}}$, where $\frac{\tilde{\alpha}-1}{2} = \min\left\{0,\frac{\alpha - 1}{2}, -s_0 + \frac{\sigma}{2} \right\}$. If $-s_0 + \frac{\sigma}{2} <\min\left\{0, \frac{\alpha - 1}{2}\right\}$, we take $s_1 = s_0 - \frac{\sigma}{2}$ and repeat the argument above. We may decrease $\sigma$ if necessary so that $\min\left\{0,\frac{\alpha - 1}{2}\right\}\neq -s_0 +k\frac \sigma 2$ for any $k=1,2,\ldots$.
After repeating the argument finite times, we obtain the estimate \eqref{grad_u_upperbound1} or \eqref{grad_u_upperbound2}.
\end{proof}

\section{Proof of Theorem \ref{main_thm}}\label{sec6}
In this section, we prove Theorem \ref{main_thm}. Without loss of generality, we assume $\mu = 1$ again throughout this section. The following ordinary differential operator
\begin{equation}\label{def_L}
Lh :=h''(r) + \left( \frac{n-2}{r} + \frac{2r}{\varepsilon + r^2} \right) h'(r) - \left( \frac{n-2}{r^2} + \frac{2}{\gamma} \frac{1}{\varepsilon+r^2} \right)  h(r)
\end{equation}
will play an important role in our proof. First, we prove the following ODE lemma.

\begin{lemma}
\label{ODE_lemma}
For $\varepsilon > 0$, $\gamma > 1$, $n \ge 2$, let $L$ be the operator defined as in \eqref{def_L}. There exists a unique solution $h \in C([0,1]) \cap C^\infty((0,1])$ of
\begin{equation}
\label{Homogeneous_ODE}
Lh = 0, \quad 0 < r <1
\end{equation}
satisfying  $h(0) = 0$ and $h(1) = 1$.
Moreover, there exist a positive constant $C(\varepsilon)$ depending only on $n$, $\gamma$, and $\varepsilon$, such that
\begin{equation}
\label{h_bounds}
r  < h(r) < \min\{ C(\varepsilon)r, r^\alpha\} \quad \mbox{for}~~0 < r < 1,
\end{equation}
where $\alpha$ is given by \eqref{alpha_new} and $h$ is strictly increasing in $[0,1]$.
\end{lemma}

\begin{proof}
Because of the singularity of $L$ at $r = 0$, we will construct the solution $h$ through the following approximating process. For $0 < a < 1$, let $h_a \in C^2([a,1])$ be the solution of $Lh_a = 0$ in $(a, 1)$ satisfying $h_a(a) = a$ and $h_a(1) = 1$. By similar computations as in Lemma \ref{lemma_v1}, we have
\begin{equation*}
Lr = \frac{2r}{\varepsilon + r^2} \left( 1 - \frac{1}{\gamma} \right) > 0 \quad\quad \mbox{for}~r\in(0,1),
\end{equation*}
and
\begin{equation}\label{Lr^alpha}
\begin{aligned}
Lr^\alpha =  - 2 \varepsilon \left( \alpha - \frac{1}{\gamma} \right) \frac{ r^{\alpha - 2}}{\varepsilon+r^2} < 0  \quad\quad\mbox{for}~r\in(0,1).
\end{aligned}
\end{equation}
Note that $h_a(a)=a<a^\alpha$ since $\alpha\in(0,1)$.
By the maximum principle and the strong maximum principle,
$$
r < h_a(r) < r^\alpha, \quad a < r < 1.
$$
Sending $a \to 0$ along a subsequence, $h_a \to h$ in $C^2_{\text{loc}}((0,1])$ for some $h \in C([0,1]) \cap C^\infty((0,1])$ satisfying $r \le h(r) \le r^\alpha$, $Lh = 0$ in $(0,1)$, and $h(0) = 0$. By the strong maximum principle,
$$
r < h(r) < r^\alpha, \quad 0 < r < 1.
$$
Next, we show the upper bound $h(r) < C(\varepsilon)r$. Let $v = r - r^{3/2}/2$, by a direct computation,
$$
\left| Lv + \frac{1}{4} \Big( n - \frac{1}{2} \Big) r^{-\frac{1}{2}}  \right| \le \frac{2}{\varepsilon} \Big( 1-\frac{1}{\gamma} \Big) r + \frac{1}{\varepsilon} \Big( 3-\frac{2}{\gamma} \Big) r^{\frac{3}{2}}.
$$
Hence $Lv < 0$ in $(0, r_0(\varepsilon))$, for some small $r_0(\varepsilon)$. Recall that $Lh = 0$ and $h < r^\alpha$ in $(0, r_0(\varepsilon))$, $h(0) = v(0) = 0$. By the maximum principle, we have 
$$
h \le \frac{h(r_0(\varepsilon))}{v(r_0(\varepsilon))} v \le C(\varepsilon)r\quad \text{in}\,\, (0,r_0(\varepsilon)),
$$ 
where $C(\varepsilon) = \frac{r_0^\alpha(\varepsilon)}{v(r_0(\varepsilon))}$. This concludes the proof of \eqref{h_bounds}.

Now, we show the uniqueness of $h$. Let $g\in C([0,1]) \cap C^\infty((0,1])$ be another solution of \eqref{Homogeneous_ODE} in $(0,1)$ satisfying $g(0) = 0$ and $g(1) = 1$. Then $w(r) := g(r) / h(r)$ satisfies
$$
(G w')' = 0, \quad 0<r<1,
$$
where $G = h^2 r^{n-2}(\varepsilon + r^2)$. Therefore, for some constants $C_0$ and $C_1$, we have
$$
g(r) = h(r) w(r) = h(r) \int_{r}^1 \frac{C_0}{h^2(s) s^{n-2} (\varepsilon + s^2)} \, ds + C_1 h(r), \quad 0 < r <  1.
$$
By \eqref{h_bounds}, we have
$$
h(r) \int_{r}^1 \frac{1}{h^2(s) s^{n-2} (\varepsilon + s^2)} \, ds \ge  \frac{r}{C(\varepsilon)^2} \int_{r}^1 \frac{1}{s^{n} (\varepsilon + s^2)} \, ds.
$$
When $n = 2$, by the L'Hospital rule,
$$
\frac{r}{C(\varepsilon)^2} \int_{r}^1 \frac{1}{s^{2} (\varepsilon + s^2)} \, ds = \frac{r}{C(\varepsilon)^2\varepsilon}\int_{r}^1 \Big( \frac{1}{s^2} - \frac{1}{\varepsilon + s^2} \Big) \, ds \to \frac{1}{C(\varepsilon)^2 \varepsilon} > 0
$$
as $r \to 0$. When $n \ge 3$, 
$$
\frac{r}{C(\varepsilon)^2} \int_{r}^1 \frac{1}{s^{n} (\varepsilon + s^2)} \, ds  \to + \infty
$$
as $r \to 0$. Since $h(0) = g(0) = 0$ and $h(1) = g(1) = 1$, we have $C_0 = 0$, $C_1 = 1$. Therefore, $g \equiv h$.

Finally, we show that $h$ is strictly increasing in $(0,1)$. If not, there exists an $r_* \in (0,1)$ such that $h'(r_*) = 0$ and $h''(r_*) \le 0$. Since $h(r_*) > 0$, we have $Lh(r_*) < 0$, which leads to a contradiction.
\end{proof}  

Next, we prove a lower bound of $h$ by constructing a subsolution.

\begin{lemma}\label{lemma_subsolution}
For $\varepsilon > 0$, $\gamma > 1$, $n \ge 2$, let $h \in C([0,1]) \cap C^\infty((0,1])$ be the solution to \eqref{Homogeneous_ODE}. There exists a positive constant $c$, depending only on $n$ and $\gamma$, such that
\begin{equation}\label{h_lowerbound}
h(r) > \Big(r^\alpha - (c\sqrt\varepsilon)^{\alpha} \Big)_+, \quad 0 < r <1,
\end{equation}
where $\alpha$ is given by \eqref{alpha_new}.
\end{lemma}
\begin{proof}
Denote $v(r) = \Big(r^\alpha - (c\sqrt\varepsilon)^{\alpha} \Big)_+$ for some positive constant $c$ to be determined. When $r > c\sqrt\varepsilon$, by \eqref{Lr^alpha}, we know
\begin{align*}
 Lv =& - 2 \varepsilon \left( \alpha - \frac{1}{\gamma} \right) \frac{ r^{\alpha - 2}}{\varepsilon+r^2} + \left( \frac{n-2}{r^2} + \frac{2}{\gamma} \frac{1}{\varepsilon+r^2} \right) (c\sqrt\varepsilon)^{\alpha} \\
 \ge & \frac{c^{\alpha-2}\varepsilon^{\frac{\alpha}{2}}}{\varepsilon + r^2} \left( -2\alpha + \frac{2}{\gamma} + c^2\left( n-2 + \frac{2}{\gamma} \right) \right).
\end{align*}
Therefore, for any $\gamma > 1$, $n \ge 2$, we can find a positive $c$ such that $Lv > 0$ when $r> c\sqrt\varepsilon$. Then \eqref{h_lowerbound} follows from the maximum principle.
\end{proof}

Now we are ready to prove Theorem \ref{main_thm}.

\begin{proof}[Proof of Theorem \ref{main_thm}]
Without loss of generality, we assume that $\mu = 1$. For general $\mu > 0$, we only need to replace $\varepsilon$ and $\gamma$ in the proof with $\varepsilon/\mu$ and $\mu\gamma$, respectively. We will prove \eqref{grad_u_lower_bound} for $n \ge 3$ in the following four steps. The proof for $n = 2$ is similar and simpler. A brief remark on the case $n=2$ will be provided at the end.

\textbf{Step 1.} By the elliptic theory, the fact that $\widetilde\Omega$ is symmetric in $x_1$, and the fact that $\varphi$ is odd in $x_1$, we know that $u$ is smooth, $|u| \le 5$, and $u$ is odd in $x_1$. Take the change of variables \eqref{x_to_y} in $\Omega_1$, let $v(y) = u(x)$ and $\bar{v}$ be the vertical average as in \eqref{v_bar_def}. By Lemma \ref{lemma_reduction}, we know that $\bar{v}$ satisfies the equation \eqref{equation_v_bar}
$$
\dv((\varepsilon+|y'|^2)\nabla \bar v) - 2\gamma^{-1} \bar{v}(y') = \dv F + G \quad \mbox{in}~B_1 \subset \bR^{n-1}.
$$
By \eqref{gradient-1/2} with $U_1 = U_2 = 0$ and \eqref{grad_n_v_bound}, we know that
$$
|\nabla_{y'} v(y)| \le C(\varepsilon +|y'|^2)^{-1/2}, \quad |\partial_n v(y)| \le C (\varepsilon + |y'|^2) \varepsilon^{ - 1}\quad \mbox{in}~ Q_{1,\varepsilon},
$$
where $Q_{1,\varepsilon}$ is defined in \eqref{Q_s_t}. Since both $D_1$ and $D_2$ are balls, $e^i$ from \eqref{a_ij_formula} can be bounded by $C|y'|^4$. Therefore by \eqref{FG_bound},  for $|y'| < 1$ we have
\begin{equation}\label{FG_bound_2}
\left\{
\begin{aligned}
|F(y')| &\le C\Big(|y'|(\varepsilon + |y'|^2) + |y'|^4 (\varepsilon +|y'|^2)^{-1/2}\Big), \\
 |G(y')| &\le C(\varepsilon + |y'|^2).
\end{aligned}\right.
\end{equation}
Similar to the proof of Lemma \ref{lemma_v1}, we denote $Y_{k,i}$ to be a $k$-th degree normalized spherical harmonics so that $\{Y_{k,i}\}_{k,i}$ forms an orthonormal basis of $L^2(\bS^{n-2})$, $Y_{1,1}$ to be the one after normalizing $\left.y_1\right|_{\bS^{n-2}}$, and we write $y' = (r ,\xi)$ in the polar coordinate. Since $\bar v$ is odd with respect to $y_1$, and in particular $\bar{v}(0) = 0$, we have the following decomposition
\begin{equation}
\label{v_bar_expansion}
\bar{v}(y') = V_{1,1}(r)Y_{1,1}(\xi) + \sum_{k=2}^\infty \sum_{i=1}^{N(k)} V_{k,i}(r)Y_{k,i}(\xi), \quad x' \in B_1\setminus\{0\},
\end{equation}
where $V_{k,i}(r) = \int_{\bS^{n-2}} \bar{v}(r,\xi) Y_{k,i}(\xi) \, d\xi$ and $V_{k,i} \in C([0,1)) \cap C^\infty((0,1))$. Since $\bar{v}(0) = 0$ and $\varepsilon + |y'|^2$ is independent of $\xi$, we know that $V_{1,1}$ satisfies $V_{1,1}(0) = 0$, and
\begin{equation*}
LV_{1,1} = H(r), \quad 0 < r <1,
\end{equation*}
where $L$ is the differential operator defined in \eqref{def_L},
\begin{align*}
H(r) &= \int_{\bS^{n-2}} \frac{(\dv F + G) Y_{1,1}(\xi)}{\varepsilon+r^2} \, d\xi = \int_{\bS^{n-2}} \frac{\partial_r F_r + \frac{1}{r} \nabla_\xi F_\xi + G}{\varepsilon + r^2} Y_{1,1}(\xi) \, d\xi\\
&= \partial_r \left(\int_{\bS^{n-2}} \frac{F_r}{\varepsilon + r^2} Y_{1,1}(\xi) \, d\xi \right) + \int_{\bS^{n-2}} \frac{2rF_r Y_{1,1}}{(\varepsilon+r^2)^2} - \frac{F_\xi \nabla_\xi Y_{1,1}}{r(\varepsilon + r^2)} + \frac{G Y_{1,1}}{\varepsilon+ r^2}\, d\xi\\
&=: A'(r) + B(r), \quad 0 < r < 1.
\end{align*}
Hence $A(r),B(r) \in C^1([0,1))$ satisfy, in view of \eqref{FG_bound_2}, that
\begin{equation}
\label{AB_bounds}
|A(r)| \le C(n,\gamma)r, \quad |B(r)| \le C(n,\gamma), \quad 0 < r < 1.
\end{equation}

\textbf{Step 2.}
We will prove, for some $\varepsilon$-dependent constant $C_1(\varepsilon)$ and positive $\varepsilon$-independent constant $C_2$, that
\begin{equation}
\label{V_11_formula}
|V_{1,1}(r) - C_1 (\varepsilon) h(r)| \le  C_2 r^{1+\alpha}, \quad 0<r<1,
\end{equation}
where $h(r)$ is the solution of \eqref{Homogeneous_ODE} satisfying $h(0)=h(1) = 1$. We use the method of reduction of order to write down a bounded function $g$ satisfying $Lg = H$ in $(0,1)$, and then give an estimate $|g(r)| \le C_2 r^{1+\alpha}$.

Let $g = h w$ and
$$
w(r) := \int_0^r\frac{1}{h^2(s) s^{n-2} (\varepsilon + s^2)} \int_0^s h(\tau) \tau^{n-2} (\varepsilon + \tau^2) H(\tau) \, d\tau ds, \quad 0 < r < 1.
$$
By a direct computation,
$$
Lg = L(hw) = h w'' + \left[ 2h' + \left( \frac{n-2}{r} + \frac{2r}{\varepsilon + r^2} \right)h \right]w' = \frac{h}{G}(Gw')' = H,
$$
where $G= h^2 r^{n-2}(\varepsilon + r^2)$. By integration by parts, \eqref{AB_bounds}, and the fact that $h > 0$ and $h' \ge 0$, we can estimate
\begin{align*}
 &\left| \int_0^s h(\tau) \tau^{n-2} (\varepsilon + \tau^2) A'(\tau) \, d\tau \right| + \left| \int_0^s h(\tau) \tau^{n-2} (\varepsilon + \tau^2) B(\tau) \, d\tau \right| \\
\le& \left|\int_0^s h' \tau^{n-2} (\varepsilon + \tau^2) A(\tau) \, d\tau \right| + C  h(s)s^{n-1}(\varepsilon + s^2)\\
\le & C s^{n-1}(\varepsilon + s^2) \int_{0}^s h'(\tau) \, d\tau+ C  h(s)s^{n-1}(\varepsilon + s^2)\\
\le & C  h(s)s^{n-1}(\varepsilon + s^2), \quad\quad  0 < r < 1.
\end{align*}
Therefore, using \eqref{h_bounds},
$$
|g(r)| \le C  h(r) \int_0^r \frac{s}{h(s)}\, ds \le  C_2 r^{1+\alpha}, \quad 0<r<1.
$$
Since $V_{1,1} - g$ is bounded and satisfies $L(V_{1,1} - g) = 0$ in $(0,1)$ and $V_{1,1}(0) - g(0) = 0$, by Lemma \ref{ODE_lemma}, there exists a constant $C_1(\varepsilon)$ such that $V_{1,1} - g = C_1(\varepsilon)h$. Hence \eqref{V_11_formula} follows.

\textbf{Step 3.} We will show that 
\begin{equation}\label{C_1_epsilon_lowerbound}
C_1 (\varepsilon) > \frac{1}{C}
\end{equation}
for some positive $\varepsilon$-independent constant $C$.

By a direct computation, $x_1 + \gamma \partial_\nu x_1 = (1-\gamma) x_1$ on $\partial D_1 \cup \partial D_2$. Therefore, $x_1$ is a subsolution of \eqref{robin} in $\{x_1 > 0\}$, and is a supersolution of \eqref{robin} in $\{x_1 < 0\}$. Hence, $u \ge x_1$ in $\{x_1 \ge 0\}$ and $u \le x_1$ in $\{x_1 \le 0\}$. Then
$$\left\{\begin{aligned}
&\bar{v}(y') \ge y_1, &&\mbox{when}~y_1\ge 0,\\ 
&\bar{v}(y') \le y_1, &&\mbox{when}~y_1\le 0.\\ 
\end{aligned}\right.
$$
Recall that $Y_{1,1}(\xi)$ is the normalization of $y_1$, we have
$$
V_{1,1}(r) = \int_{\bS^{n-2}}  \bar{v}(r,\xi) Y_{1,1}(\xi) \, d\xi \ge r, \quad 0 < r < 1.
$$
Combining with \eqref{h_bounds} and \eqref{V_11_formula}, we have
\begin{align*}
r \le V_{1,1}(r) \le C_1(\varepsilon)h(r) + \frac{1}{2}r \le C_1(\varepsilon)r^\alpha + \frac{1}{2}r \quad \forall 0 < r \le r_0,
\end{align*}
where $r_0 = (1/2C_2)^{1/\alpha}$. This implies
$$
C_1(\varepsilon) \ge \frac{1}{2}r_0^{1-\alpha}.
$$

\textbf{Step 4.} Completion of the proof of Theorem \ref{main_thm}.

It follows, in view of \eqref{h_bounds}, \eqref{V_11_formula} and \eqref{C_1_epsilon_lowerbound}, that
\begin{equation}
\label{V_11_lowerbound}
V_{1,1}(r) \ge \frac{1}{C}  h(r) - C_2 r^{1+\alpha} \ge \frac{1}{2C} h(r) , \quad 0 < r < r_1
\end{equation}
for some small $\varepsilon$-independent constant $r_1$. By \eqref{h_lowerbound},
\begin{equation}
\label{h1_lowerbound}
h(2c\sqrt{\varepsilon}) > (2^\alpha - 1)c^\alpha \varepsilon^{\alpha/2},
\end{equation}
where $c$ is the constant in Lemma \ref{lemma_subsolution}. By \eqref{v_bar_expansion}, \eqref{V_11_lowerbound}, and \eqref{h1_lowerbound}, we have
\begin{align*}
\left( \int_{\bS^{n-2}} |\bar{v}(2c\sqrt{\varepsilon},\xi)|^2 \, d\xi \right)^{1/2} &\ge |V_{1,1}(2c\sqrt{\varepsilon})| \ge \frac{1}{C} h(2c\sqrt{\varepsilon}) \ge \frac{1}{C}\varepsilon^{\alpha/2} \quad \forall 0 < \varepsilon < \varepsilon_0
\end{align*}
for some small positive $\varepsilon_0$ depending only on $n$ and $\gamma$. Then, there exists a $\xi_0 \in \bS^{n-2}$ such that $|\bar{v}(2c\sqrt{\varepsilon}, \xi_0)| \ge \frac{1}{C} \varepsilon^{\alpha/2}$. Since $\bar v$ is the vertical average of $v$, and recall that $v(y) = u(x)$, there exists an $x_n \in (-\varepsilon/2 + g(x'), \varepsilon/2 + f(x'))$ such that
\begin{equation}
\label{u_lower_bound}
|u(2c\sqrt{\varepsilon}, \xi_0, x_n)| \ge \frac{1}{C} \varepsilon^{\frac{\alpha}{2}}.
\end{equation}
Estimate \eqref{grad_u_lower_bound} follows from \eqref{u_lower_bound} and $u(0) = 0$. This concludes the proof when $n \ge 3$.

When $n = 2$, by Lemma \ref{lemma_reduction}, $\bar{v}$ satisfies the equation
$$
((\varepsilon+|y_1|^2)\bar v')'- 2\gamma^{-1} \bar{v} = F' + G\quad \mbox{in}~(-1,1),
$$
where $F$ and $G$ satisfy the same bounds as in \eqref{FG_bound_2}. By using $\bar{v}$ in place of $V_{1,1}$ and repeating steps 2-4 with minor modifications, we can establish the estimate \eqref{grad_u_lower_bound} for $n = 2$.
\end{proof}

\appendix

\section{}
In this appendix, we prove Lemma \ref{Change_of_variable_lemma}. The proof essentially follows that of \cite{DYZ23}*{Lemma 3.1}. 
\begin{proof}[Proof of Lemma \ref{Change_of_variable_lemma}.]
Following the proof of \cite{DYZ23}*{Lemma 3.1}, with \eqref{h_derivatives} in place of \cite{DYZ23}*{(3.9)}, for $y \in \overline{Q_{2r,r^2}}$, we have
$$
|D_{y'} x' - I_{(n-1)\times(n-1)}| \le Cr^2, \quad |D_{y_n} x'| \le Cr,  \quad |D_{y'}x_n| \le Cr.
$$
Since
\begin{align*}
D_{y_n} x_n = \frac{1}{2r^2} (\varepsilon + \tilde f(y') - \tilde g(y')).
\end{align*}
By \eqref{height_final}, for $y \in \overline{Q_{2r,r^2}}$, we also have
$$
\frac{1}{C} \le D_{y_n} x_n \le C.
$$
Then (a) follows by shrinking $r_0$ to be sufficiently small.

Since $h(y) = 0$ when $y = \pm r^2$, $\Phi$ maps the upper and lower boundaries of $Q_{2r, r^2}$ onto the upper and lower boundaries of $\Omega_{x_0,2r} $, respectively. Then (b) simply follows from the fact that $|h(y)| \le Cr^3$, and we can shrink $r_0$ so that $|h(y)| \le r/10$.

To verify (c), note that $u$ is smooth by the classical elliptic theory. By the chain rule, 
$$
\begin{aligned}
D_{x_k}u(x)&= D_{y_i}\tilde u(y) D_{x_k}y_i,\\
D^2_{x_k}u(x)&=D_{y_i}D_{y_j}\tilde u(y) D_{x_k}y_i D_{x_k}y_j + D_{y_i}\tilde u (y) D_{x_k}^2 y_i.
\end{aligned}
$$
For $i,j\in \{1,\ldots, n\}$,
we define $$a^{ij}:= D_{x} y_i\cdot D_{x} y_j, \quad  b^{i}:=\sum_{k=1}^n D^2_{x_k}y_i.$$
Then we have 
\begin{equation}\label{equiv:a}
    a:=(a^{ij})=(D\Phi)^{-1}((D\Phi)^{-1})^T,
\end{equation}
and $\tilde u=\tilde u(y)$ satisfies the equation
\begin{equation}\label{v_equation_1}
a^{ij}D_{y_iy_j}\tilde u+b^i D_{y_i}\tilde u=0   \quad \mbox{in }Q_{1.9r, r^2}.
\end{equation}
We then deduce the boundary condition for $\tilde u$ on $\{y_n = \pm r^2\}$. By the chain rule,
$$
D_{y_n} \tilde u(y)=D_{y} \tilde u(y) \cdot e_n = D_x u(x) \cdot D_y \Phi\,e_n,
$$
where $e_n:= (0,\ldots,0,1)^T$. 
Similar to the proof of \cite{DYZ23}*{Lemma 3.1}, using 
\begin{equation}\label{f=tildef=muf}
    f\equiv \tilde f\equiv\tilde f^\kappa, \quad g\equiv \tilde g \equiv \tilde g^\kappa \quad \mbox{on }  \{y_n=\pm r^2\}\cap \overline{Q_{2r, r^2}},
\end{equation}
we know that 
\begin{equation}\label{parallel}
\begin{aligned}
 D_y \Phi\,e_n&=\Big(-D_{y_n}h , \frac{1}{2r^2} (\varepsilon + \tilde f(y') - \tilde g(y')) 
 \Big)^T\\
 &=\frac{1}{2r^2} (\varepsilon + \tilde f(y') - \tilde g(y')) \Big(-D_{x'} \tilde f(y'),1\Big)^T
 \end{aligned}
\end{equation}
when $y_n=r^2$.
Therefore, $\tilde u$ satisfies 
\begin{equation}\label{bc_up}
    \tilde u+\gamma \big(\frac{1}{2r^2}(\varepsilon+\tilde f-\tilde g)\sqrt{1+|D_{x'}\tilde f|^2}\big)^{-1}D_{y_n}\tilde u=U_1\quad\mbox{on } \{y_n=r^2\}\cap \overline{Q_{1.9r, r^2}}.
\end{equation}
Similarly, $\tilde u$ also satisfies
\begin{equation}\label{bc_low}
\tilde u-\gamma \big(\frac{1}{2r^2}(\varepsilon+\tilde f-\tilde g)\sqrt{1+|D_{x'}\tilde g|^2}\big)^{-1}D_{y_n}\tilde u=-U_1 \quad\mbox{on } \{y_n=-r^2\}\cap \overline{Q_{1.9r, r^2}}.
\end{equation}
By \eqref{v_equation_1}, \eqref{bc_up}, \eqref{bc_low}, and \eqref{f=tildef=muf},
we know that $\tilde u$ is a solution to \eqref{tilde_u_equation}.

Next, we prove \eqref{est:a}.
From part (a), we have
$$
\frac{I}{C} \le  D_{x}y=(D\Phi)^{-1} \le C I,
$$
which directly implies that
\begin{equation*}
    \frac{I}{C} \le a \le C I,
\end{equation*}
where $a:=(a^{ij})_{n\times n}$.
By a similar proof as that of \cite{DYZ23}*{Lemma 3.1}, we know that 
$$
\Big|\frac{\partial^2 y_i}{\partial x_k \partial x_l} \Big|\le \frac{C}{r} \quad \mbox{for } i,k,l\in\{1,2,\ldots, n\}.
$$
The last three inequalities directly imply \eqref{est:a}.

Now, we prove \eqref{ortho}. 
Let $a^{-1}:=(c^{ij})_{n\times n}$. By \eqref{equiv:a}, 
$$a^{-1}= (D\Phi)^T D\Phi.$$
When $y_n=r^2$, we know that
$$\frac{\partial x'}{\partial y'}=I_{n-1},\quad \frac{\partial x'}{\partial y_n}=-D_{y_n}h, \quad \frac{\partial x_n}{\partial y'}=D_{y'}\tilde f, \quad \frac{\partial x_n}{\partial y_n}=\frac{1}{2r^2} (\varepsilon + \tilde f(y') - \tilde g(y')).$$
Therefore, by \eqref{parallel}, when $y_n= r^2$, for $j\in\{1,2,\ldots,n-1\}$, we have 
$$
\begin{aligned}
    c^{nj}=c^{jn}=(\det (D\Phi))^{-1}\Big(-D_{y_n}h^j+\frac{1}{2r^2} (\varepsilon + \tilde f(y') - \tilde g(y'))D_{y_j}\tilde f\Big)=0.
\end{aligned}
$$
Similarly, we also have
$$
 c^{nj}=c^{jn}=0 \quad \mbox{when } y_n=-r^2, \quad \mbox{for } j\in\{1,2,\ldots,n-1\}.
$$
Since $a^{-1}=(c^{ij})$, the equalities above imply \eqref{ortho}.

Finally, we prove \eqref{F_bound}.
Using \eqref{fg_0}--\eqref{def:c_2}, and \eqref{h_derivatives}, one can directly obtain 
$$
\frac{1}{C}\le F_i \le C, \quad |D_{y'} F_i|\le \frac{C}{r}.
$$
Similar to the proof of \cite{DYZ23}*{Lemma 3.1}, we have
\begin{equation}\label{f_higher}
\begin{aligned}
&|D_{y_n} D_{y'} \tilde{f}^\kappa(y')|\le Cr, \quad |D_{y'}^3 \tilde{f}^\kappa(y')|\le \frac{Cr}{r^4 - y_n^2},\\
&|D_{y_n}D_{y'}^2 \tilde{f}^\kappa (y')|\le \frac{Cr^2}{r^4 - y_n^2},\quad |D_{y_n}^2 D_{y'} \tilde{f}^\kappa (y')| 
\le \frac{Cr^3}{r^4 - y_n^2}.
\end{aligned}
\end{equation}
The inequalities above also hold with $\tilde{g}^\kappa$ in place of $\tilde{f}^\kappa$.
Direct computations using \eqref{h_derivatives} and \eqref{f_higher} yield 
$$
\begin{aligned}
&|D_{y_n} F_i|\le C\, r^2, \quad |D^2_{y'}F_i(y)|\le \frac{C\,r^2}{r^4-y_n^2},\\
&|D_{y_n}D_{y'}F_i(y)|\le  \frac{C\,r^3}{r^4-y_n^2},\quad |D^2_{y_n}F_i(y)|\le \frac{C\,r^4}{r^4-y_n^2}.
\end{aligned}
$$
Thus \eqref{F_bound} holds. The lemma is proved.
\end{proof}

\bibliographystyle{amsplain}

\end{document}